\documentclass{preprint}
\usepackage[T1]{fontenc}
\usepackage{color}
\usepackage{verbatim}
\usepackage{mathrsfs}
\usepackage{amsthm}
\usepackage{amsmath}
\usepackage{amssymb}
\usepackage{esint}

\makeatletter
\numberwithin{equation}{section}
\numberwithin{figure}{section}
\theoremstyle{plain}
\newtheorem{thm}{\protect\theoremname}[section]
  \theoremstyle{plain}
  \newtheorem{prop}[thm]{\protect\propositionname}
  \theoremstyle{plain}
  \newtheorem{lem}[thm]{\protect\lemmaname}
  \theoremstyle{plain}
  \newtheorem{cor}[thm]{\protect\corollaryname}
  \theoremstyle{plain}
  \newtheorem{conjecture}[thm]{\protect\conjecturename}
  \theoremstyle{definition}
  \newtheorem{problem}[thm]{\protect\problemname}

\usepackage{textcomp}
\usepackage[utf8]{inputenc}
\usepackage{lmodern}

\usepackage{stmaryrd}
\usepackage{bbm}

\definecolor{magenta}{RGB}{30, 0, 50}
\makeatother

  \providecommand{\conjecturename}{Conjecture}
  \providecommand{\corollaryname}{Corollary}
  \providecommand{\lemmaname}{Lemma}
  \providecommand{\problemname}{Problem}
  \providecommand{\propositionname}{Proposition}
\providecommand{\theoremname}{Theorem}

\begin{document}

\title{Smooth approximation of Yang--Mills theory on $\mathbb{R}^{2}$:
a rough path approach}

%\author{Hideyasu Yamashita \\ {\normalsize{}Division of Liberal Arts and
%Sciences, Aichi-Gakuin University} \\ {\normalsize{}email: }\texttt{\normalsize{}yamasita@dpc.aichi-gakuin.ac.jp}}

\author{Hideyasu Yamashita}
\institute{Division of Liberal Arts and Sciences, Aichi-Gakuin University
 \\ \email{yamasita@dpc.aichi-gakuin.ac.jp}}
\maketitle
\newcommand{\hidable}[3]{#2}
\newcommand{\hidea}[1]{{#1}}
\newcommand{\hideb}[1]{{#1}}
\newcommand{\hidec}[1]{{#1}}
\newcommand{\hidep}[1]{{#1}}
\renewcommand{\hidec}[1]{}
\renewcommand{\hidep}[1]{}

\newcommand{\thlab}[1]{{\tt [#1]}}
\renewcommand{\thlab}[1]{}

\begin{comment}
basic \rule[0.5ex]{0.7\columnwidth}{2pt}
\end{comment}

\global\long\def\N{\mathbb{N}}
\global\long\def\C{\mathbb{C}}
\global\long\def\Z{\mathbb{Z}}
 \global\long\def\R{\mathbb{R}}
 \global\long\def\im{\mathrm{i}}

\global\long\def\di{\partial}
 \global\long\def\d{{\rm d}}

\global\long\def\ol#1{\overline{#1}}
\global\long\def\ul#1{\underline{#1}}
\global\long\def\ob#1{\overbrace{#1}}

\global\long\def\ov#1{\overline{#1}}
 %
\begin{comment}
\global\long\def\bbD{\mathbb{D}}
\end{comment}

\global\long\def\then{\Rightarrow}
 \global\long\def\Then{\Longrightarrow}

\begin{comment}
Greek \rule[0.5ex]{0.7\columnwidth}{2pt}
\end{comment}

\global\long\def\al{\alpha}
\global\long\def\de{\delta}
 \global\long\def\ep{\epsilon}
 \global\long\def\la{\lambda}
 \global\long\def\io{\iota}
 \global\long\def\th{\theta}
\global\long\def\si{\sigma}
 \global\long\def\om{\omega}

\global\long\def\De{\Delta}
 \global\long\def\Th{\Theta}
 \global\long\def\Om{\Omega}

\global\long\def\brho{\boldsymbol{\rho}}
\global\long\def\bDelta{\boldsymbol{\Delta}}
 \global\long\def\bmu{\boldsymbol{\mu}}
 \global\long\def\bchi{\boldsymbol{\chi}}
 \global\long\def\bOm{\boldsymbol{\Omega}}

\begin{comment}
calligraphic \rule[0.5ex]{0.7\columnwidth}{2pt}
\end{comment}

\global\long\def\cA{\mathcal{A}}
\global\long\def\cB{\mathcal{B}}
 \global\long\def\cC{\mathcal{C}}
 \global\long\def\cD{\mathcal{D}}
\global\long\def\cE{\mathcal{E}}
 \global\long\def\cF{\mathcal{F}}
 \global\long\def\cG{{\cal G}}
 \global\long\def\cH{\mathcal{H}}
 \global\long\def\cI{\mathcal{I}}
 \global\long\def\cJ{\mathcal{J}}
\global\long\def\cK{\mathcal{K}}
 \global\long\def\cL{\mathcal{L}}
 \global\long\def\cM{\mathcal{M}}
 \global\long\def\cN{\mathcal{N}}
 \global\long\def\cO{\mathcal{O}}
 \global\long\def\cP{\mathcal{P}}
 \global\long\def\cQ{\mathcal{Q}}
 \global\long\def\cR{\mathcal{R}}
 \global\long\def\cS{\mathcal{S}}
 \global\long\def\cT{\mathcal{T}}
 \global\long\def\cU{\mathcal{U}}
 \global\long\def\cV{\mathcal{V}}
 \global\long\def\cW{\mathcal{W}}
\global\long\def\cX{\mathcal{X}}
 \global\long\def\cY{\mathcal{Y}}
 \global\long\def\cZ{\mathcal{Z}}

\begin{comment}
script \rule[0.5ex]{0.7\columnwidth}{2pt}
\end{comment}

\global\long\def\scA{\mathscr{A}}
\global\long\def\scB{\mathscr{B}}
\global\long\def\scC{\mathscr{C}}
\global\long\def\scD{\mathscr{D}}
 \global\long\def\scE{\mathscr{E}}
 \global\long\def\scF{\mathscr{F}}
 \global\long\def\scI{\mathscr{I}}
 \global\long\def\scL{\mathscr{L}}
\global\long\def\scM{\mathscr{M}}
\global\long\def\scP{\mathscr{P}}
\global\long\def\scS{\mathscr{S}}
 \global\long\def\scT{\mathscr{T}}

\global\long\def\bbA{\mathbb{A}}
 \global\long\def\bbB{\mathbb{B}}
\global\long\def\bbK{\mathbb{K}}
\global\long\def\bbQ{\mathbb{Q}}
\global\long\def\bbT{\mathbb{T}}
 \global\long\def\bbU{\mathbb{U}}
\global\long\def\bbX{\mathbb{X}}
\global\long\def\bbY{\mathbb{Y}}
\global\long\def\bbW{\mathbb{W}}

\global\long\def\bB{\mathbf{B}}
 \global\long\def\bS{\boldsymbol{S}}
 \global\long\def\bU{{\bf U}}
\global\long\def\bX{\mathbf{X}}
\global\long\def\bY{\mathbf{Y}}
\global\long\def\bW{\mathbf{W}}

\global\long\def\fg{\mathfrak{g}}
\global\long\def\fs{\mathfrak{s}}
 \global\long\def\fD{\mathfrak{D}}
 \global\long\def\fK{\mathfrak{K}}
 \global\long\def\frad{\mathfrak{d}}

\global\long\def\hM{\hat{M}}

\global\long\def\rM{\mathrm{M}}
\global\long\def\fP{\mathfrak{P}}
\global\long\def\prj{\mathfrak{P}}

\begin{comment}
color \rule[0.5ex]{0.7\columnwidth}{2pt}
\end{comment}

\global\long\def\magenta#1{{\color{magenta}#1}}

\begin{comment}
\global\long\def\symb#1{{\color{red}#1}}
\end{comment}
{} \global\long\def\symb#1{#1}

\global\long\def\emhrb#1{\text{{\color{red}\huge{\bf #1}}}}
 \global\long\def\sy#1{{\color{blue}#1}}

\newcommand{\usuji}{\color[rgb]{0.7,0.4,0.4}} \newcommand{\usu}{\color[rgb]{0.5,0.2,0.1}}
\newenvironment{Usuji} {\begin{trivlist}   \item \usuji }  {\end{trivlist}}
\newenvironment{Usu} {\begin{trivlist}   \item \usu }  {\end{trivlist}} 

\newcommand{\term}[1]{{\em #1}}

\begin{comment}
roman \rule[0.5ex]{0.7\columnwidth}{2pt}
\end{comment}

\global\long\def\supp{{\rm supp}}
\global\long\def\dom{\mathrm{dom}}
\global\long\def\ran{\mathrm{ran}}
 \global\long\def\leng{\text{{\rm leng}}}
 \global\long\def\diam{\text{{\rm diam}}}
 \global\long\def\Leb{\text{{\rm Leb}}}
 \global\long\def\meas{\text{{\rm meas}}}
\global\long\def\sgn{{\rm sgn}}
 \global\long\def\Tr{{\rm Tr}}
 \global\long\def\spec{{\rm spec}}
 \global\long\def\Ker{{\rm Ker}}
 \global\long\def\Lip{{\rm Lip}}
 \global\long\def\HS{{\rm HS}}

\begin{comment}
probability notations \rule[0.5ex]{0.7\columnwidth}{2pt}
\end{comment}

\global\long\def\Prob{\mathbb{P}}
\global\long\def\Var{\mathrm{Var}}
\global\long\def\Cov{\mathrm{Cov}}
\global\long\def\Ex{\mathbb{E}}
 %
\begin{comment}
\global\long\def\aei{\text{ a.e.}\iota}
\end{comment}
{} %\newcommand{\F}{\mathbf{F}}
\global\long\def\Ae{{\rm a.e.}}
 \global\long\def\samples{\bOm}

\begin{comment}
Friz-Victoir notations \rule[0.5ex]{0.7\columnwidth}{2pt}
\end{comment}

\global\long\def\var{\textrm{{\rm var}}}
\global\long\def\hvar{\textrm{{\rm -var}}}
 \global\long\def\pvar{p\textrm{{\rm -var}}}

\begin{comment}
\global\long\def\Hol{\text{{\rm Höl}}}
\end{comment}
{} \global\long\def\Hol{\text{{\rm Höl}}}
 %
\begin{comment}
\global\long\def\hHol{\text{{\rm -Höl}}}
\end{comment}
{} \global\long\def\hHol{\text{{\rm -Höl}}}
 %
\begin{comment}
\global\long\def\pHol{1/p\text{{\rm -Höl}}}
\end{comment}
{} \global\long\def\pHol{1/p\text{{\rm -Höl}}}

\global\long\def\frakt{\mathfrak{t}}

\global\long\def\rpvar{\mathfrak{p}}
 \global\long\def\rpHol{\mathfrak{h}}

\begin{comment}
mathop: \rule[0.5ex]{0.7\columnwidth}{2pt}
\end{comment}

\newcommand{\slim}{\mathop{\mbox{s-lim}}} %
\begin{comment}
\global\long\def\slim{\text{{\rm s-lim}}}
\end{comment}
\begin{comment}
 \global\long\def\slim{\text{{\rm s-lim}}}
\end{comment}

\newcommand{\wlim}{\mathop{\mbox{w-lim}}}

%\newcommand{\limsub}{\mathop{\mbox{\rm lim-sub}}}

\begin{comment}
misc \rule[0.5ex]{0.7\columnwidth}{2pt}
\end{comment}

\global\long\def\upha{\upharpoonright}

\global\long\def\bOne{{\bf 1}}

\global\long\def\ket#1{|#1\rangle}
 \global\long\def\bra#1{\langle#1|}

\global\long\def\Disk{\mathbb{D}^{2}}
\global\long\def\hcG{\hat{\mathcal{G}}}
\global\long\def\sfC{\mathsf{C}}

\begin{comment}
\global\long\def\crv{\mathsf{c}}
\global\long\def\Crv{\mathsf{C}}
\end{comment}
{} \global\long\def\crv{\mathfrak{c}}
\global\long\def\Crv{\mathfrak{C}}
 \global\long\def\gE{\mathrm{e}}
 \global\long\def\Rot{{\rm Rot}}

\global\long\def\lll{|||}
 %
\begin{comment}
\global\long\def\III{\interleave}
\end{comment}

\newcommand{\iiia}[1]{{\left\vert\kern-0.25ex\left\vert\kern-0.25ex\left\vert #1
  \right\vert\kern-0.25ex\right\vert\kern-0.25ex\right\vert}}

\global\long\def\iii#1{\iiia{#1}}

\global\long\def\chchi{\check{\bchi}}
\global\long\def\chrho{\check{\brho}}
\begin{comment}
\global\long\def\Sf{\Xi}
\global\long\def\chSf{\check{\Sf}}
 
\end{comment}

\global\long\def\nmat{n_{{\rm mat}}}
 \global\long\def\moll{\cM}
 \global\long\def\besovorder{\mathfrak{s}}

\global\long\def\ptrans{{\rm \mathscr{U}}}
 \global\long\def\lift{{\rm lift}}

\global\long\def\bbm#1{\mathbbm{#1}}

\begin{comment}
\global\long\def\bbm#1{{\color{blue}{\color{cyan}\mathsf{#1}}}}
\end{comment}

\global\long\def\Sn{{\rm sig}}
 \global\long\def\gtrans{\cG}
 \global\long\def\Mat{{\rm Mat}}

\global\long\def\arc{\mathsf{c}}
 \global\long\def\Lasso{{\rm Lasso}}
 \global\long\def\rectlikedomain{\cR}
 \global\long\def\dissect{{\rm diss}}
 \global\long\def\totte{\gamma}

\global\long\def\bpi{\boldsymbol{\pi}}
 \global\long\def\axial{{\bf e}}

\newcommand{\displabel}[1]{\texttt{\textup{\tiny{}lab\%#1\%}}}

\sloppy
\begin{abstract}
In the context of rough path theory (RPT), the theories of Hairer
(2014) and Gubinelli--Imkeller--Perkowski (2015) (GIP theory) gave
new methods  for construction of  $\Phi_{3}^{4}$ model. Roughly, their results state
that a quantum field in a $\Phi_{3}^{4}$ model can be smoothly approximated.
Consider the following question: Can RPT be applied to quantum Yang--Mills
(YM) gauge field theories to show that any YM theory can be smoothly
approximated? In this paper we consider this problem in the simplest
case of Euclidean YM theory, i.e. YM on $\mathbb{R}^{2}$ with the
usual Euclidean metric, as a test case. We prove that a (quantum)
$SU(n)$ YM theory on $\mathbb{R}^{2}$ in axial gauge can be smoothly
approximated for some class of Wilson loops.
 While our study is inspired by the theories of Hairer and GIP, instead
we use the RPT framework of Friz--Victoir (2010) in proving the theorem.
\end{abstract}

\tableofcontents{}

\section{Introduction}

In the context of rough path theory (e.g. \cite{FV10b,FH14}), the
theory of regularity structure of Hairer \cite{Hai14}, and that of
paracontrolled distributions of Gubinelli, Imkeller and Perkowski
(GIP theory) \cite{GIP15} gave new methods of construction of models
of quantum scalar fields, including the $\Phi_{3}^{4}$ model \cite{CC13,Hai14,Hai15,MW16,MWX16}.
Their results are summarized very roughly in one sentence: A quantum
field in a $\Phi_{3}^{4}$ model, which is represented by a distribution-valued
random variable, can be approximated by smooth fields, which are $C^{\infty}$-vauled
random variables. Thus the following natural (and naive) questions
arise: Can these methods be applied to quantum Yang--Mills (YM) gauge
field theories to show that any YM theory can be smoothly approximated?
More generally, can the notion of `rough gauge field' be rigorously
established? 

In this paper we consider this problem in the simplest case of Euclidean
YM theory, i.e. YM on $\R^{2}$ with the usual Euclidean metric, as
a test case. Our main result (Theorem \ref{thm:main-summ}) states
that a (quantum) $SU(n)$-YM theory on $\R^{2}$ in axial gauge can
be smoothly approximated; More precisely, it is stated as follows:
Let $\fg=\mathfrak{su}(n)$ be the Lie algebra of $G=SU(n)$, and
$\Omega^{1}(\R^{2},\fg)$ the space of smooth $\fg$-valued 1-forms
on $\R^{2}$. For a curve $\crv:\R\to\R^{2}$ and a 1-form $A\in\Omega^{1}(\R^{2},\fg)$,
let $\symb{\ptrans_{\crv,A}(t)}\in G$ $(t\in\R)$ denote the parallel
transport along $\crv$. Suppose that a set of the curves $\{\crv^{i}:i\in\N\}$
satisfy some regularity conditions. Then there exists a probability
space $(\samples,\Prob)$ and a sequence of $\Omega^{1}(\R^{2},\fg)$-valued
random variables $A^{(n)}$ such that 
\[
\Prob\Bigl[\forall i\in\N,\ \ptrans_{\crv^{i}}:=\lim_{n\to\infty}\ptrans_{\crv^{i},A^{(n)}}\text{ (uniform) }\in C([0,1],G)\Bigr]=1,\ 
\]
and furthermore the set of the $G$-valued random variables $\{\ptrans_{\crv^{i}}\}_{i\in\N}$
obeys the law the Wilson loops in  the YM theory on $\R^{2}$. Note
that this statement itself does not contain any term or notion specific
to rough path theory (including the theories of Hairer and GIP). However,
to prove the theorem, we shall make heavy use of rough path theory,
as well as the Littlewood--Paley theory of Besov spaces, in this paper.
While our study is inspired by the theories of GIP and regularity
structure, we work in the framework of \cite{FV10b}, without those
theories. 

While YM on $\R^{2}$ is called `trivial' in the physical literature
since this is a sort of free field theory in the sense that it does
not describe any interaction, we find that this theory has highly
`nontrivial' aspects in the mathematical point of view; Although the
above theorem can be viewed as a partial positive answer for the above
questions, our result is yet too weak to establish the theory of `rough
gauge fields.' See Conjecture  \ref{conj:main}. 

For the rigorous formulations of (Euclidean) quantum YM theories on
a 2-dimensional Riemannian manifold, we refer to Driver \cite{Dri89},
Sengupta \cite{Sen92,Sen93,Sen97} and L\'evy \cite{Lev03}.

\section{Littlewood--Paley theory and Besov space}

For a general introduction to Besov spaces with the Littlewood–Paley
theory, we refer to \cite{BCD11,Gra09} (see also Appendix of \cite{GIP15}),
and for Besov (and Sobolev) spaces \textit{without} the Littlewood–Paley
theory, we refer to \cite{Tar07}. 

Let $\cF u=\hat{u}$ denotes the Fourier transform of $u$:

\[
\symb{\cF u(z)}=\symb{\hat{u}}(z):=\int_{\R^{d}}e^{-\im\left\langle z,x\right\rangle }u(x)dx,\ 
\]
so that $\check{u}(z):=\cF^{-1}u(z)=(2\pi)^{-d}\cF u(-z)$. We consider
only the case where $d=2$.

Following Grafakos \cite{Gra09}, we fix a radial $C^{\infty}$ function
$\brho=\brho_{0}$ on $\R^{2}$ such that 
\begin{align*}
 & \brho_{0}\ge0,\ {\rm supp}\brho_{0}\subset\left\{ \xi:1-\frac{1}{7}\le|\xi|\le2\right\} \\
 & 1\le|\xi|\le2-\frac{2}{7}\Longrightarrow\brho_{0}(\xi)=1\\
 & 1\le|\xi|\le4-\frac{4}{7}\Longrightarrow\brho_{0}(\xi)+\brho_{0}(\xi/2)=1
\end{align*}
so that $\sum_{j\in\Z}\brho_{0}(2^{-j}\xi)=1$ for $\xi\in\R^{2}\setminus\{0\}$.
We also define $\symb{\bchi=\bchi_{0}}$ so that
\[
\symb{\bchi_{0}(\xi)}:=\sum_{j\le-1}\brho_{0}(2^{-j}\xi)\text{ if }\xi\neq0,\ \ \bchi_{0}(\xi)=1\text{ if }\xi=0.\quad
\]
Set
\[
\symb{\brho_{-1}}:=\bchi,\quad\symb{\brho_{j}}:=\brho_{0}(2^{-j}\cdot),\quad j\ge0,
\]
so that $\sum_{j\ge-1}\brho_{j}=1$, and set
\[
\symb{\bchi_{j}}:=\bchi_{0}(2^{-j}\cdot)=\sum_{i=-1}^{j-1}\brho_{i},\quad j\ge0
\]
Define the \term{Littlewood–Paley operators} $\bDelta_{j}$ and
$\bS_{j}$ by
\begin{align*}
\symb{\bDelta_{j}}u: & =\cF^{-1}(\brho_{j}\cF u)=\chrho_{j}*u,\quad j\ge-1,\\
\symb{\bS_{j}}u: & =\sum_{i=-1}^{j-1}\bDelta_{i}u=\chchi_{j}*u.
\end{align*}
For $p,q\in[1,\infty]$ and $\besovorder\in\R$, the \term{Besov space}
$\symb{B_{p,q}^{\besovorder}=B_{p,q}^{\besovorder}(\R^{d},\R^{n})}\subset\scS'(\R^{d},\R^{n})$
is defined by
\begin{align*}
 & \symb{B_{p,q}^{\besovorder}(\R^{d},\R^{n})}\\
 & \quad:=\biggl\{ u\in\scS'(\R^{d},\R^{n}):\symb{\|u\|_{B_{p,q}^{\besovorder}}}:=\biggl(\sum_{j\ge-1}\left(2^{j\besovorder}\left\Vert \Delta_{j}u\right\Vert _{L^{p}}\right)^{q}\biggr)^{1/q}<\infty\biggr\}.
\end{align*}
The \term{Lipschitz space} $\symb{\Lip^{\besovorder}=\Lip^{\besovorder}(\R^{d},\R^{n})}$
is defined by
\begin{align*}
\symb{\Lip^{\besovorder}(\R^{d},\R^{n})}: & =B_{\infty,\infty}^{\besovorder}(\R^{d},\R^{n})\ \\
 & =\left\{ u\in\scS'(\R^{d},\R^{n}):\symb{\|u\|_{B_{\infty,\infty}^{\besovorder}}}:=\sup_{j\ge-1}\left(2^{j\besovorder}\|\Delta_{j}u\|_{L^{\infty}}\right)<\infty\right\} 
\end{align*}
The space $B_{p,p}^{\besovorder}(\R^{d},\R^{n})$ is written as $\symb{W^{\besovorder,p}(\R^{d},\R^{n})}$,
often called the \term{Sobolev space}.

For $h\in\R^{d}$, let $\symb{\tau_{h}}$ denote the translation operator 

\begin{equation}
(\symb{\tau_{h}u})(x):=u(x+h)\ \label{eq:def:tau_h}
\end{equation}
The following proposition will be used later.
\begin{prop}
{\rm (e.g. \cite[Lemma 35.1]{Tar07})} \label{thm:B_(p,infty)-character}Let
$0<\besovorder<1$ and $1\le p\le\infty$. Define the seminorm $\symb{\left|\cdot\right|_{B_{p,\infty}^{\besovorder}}'}$
and the norm $\symb{\left\Vert \cdot\right\Vert _{B_{p,\infty}^{\besovorder}}'}$
by
\[
\symb{\left|u\right|_{B_{p,\infty}^{\besovorder}}'}:=\sup_{h\in\R^{d}\setminus\{0\}}\frac{\left\Vert u-\tau_{h}u\right\Vert _{L^{p}}}{|h|^{\besovorder}},\quad\symb{\left\Vert u\right\Vert _{B_{p,\infty}^{\besovorder}}'}:=\left\Vert u\right\Vert _{L^{p}}+\left|u\right|_{B_{p,\infty}^{\besovorder}}'.\ 
\]
Then $u\in B_{p,\infty}^{\besovorder}(\R^{d},\R^{n})$ if and only
if $\left\Vert u\right\Vert _{B_{p,\infty}^{\besovorder}}'<\infty$.
Moreover the norms $\left\Vert \cdot\right\Vert _{B_{p,\infty}^{\besovorder}}'$
and $\left\Vert \cdot\right\Vert _{B_{p,\infty}^{\besovorder}}$ are
equivalent.
\end{prop}

\section{Lie algebra valued white noise}

Fix $\nmat\in\N$ and let $\symb{\Mat}:=\Mat(\nmat,\C)\cong\R^{2\nmat^{2}}$,
equipped with the Hilbert--Schmidt inner product

\[
\symb{\left\langle X,Y\right\rangle =\left\langle X,Y\right\rangle _{\HS}}:=\Tr X^{*}Y,\quad X,Y\in\Mat,
\]
and the norm $\symb{\left\Vert X\right\Vert _{\HS}}:=\left\langle X,X\right\rangle _{\HS}^{1/2}$.
Let $G:=SU(\nmat)\subset\Mat$, and $\fg:=\mathfrak{su}(\nmat)\subset\Mat$,
the Lie algebra of $G$. We define the inner product $\left\langle \cdot,\cdot\right\rangle _{\fg}$
on $\fg$ by $\left\langle X,Y\right\rangle _{\fg}:=\left\langle X,Y\right\rangle _{\HS}$.
Note that $\left\langle \cdot,\cdot\right\rangle _{\fg}$ is proportional
to the Killing form on $\fg=\mathfrak{su}(\nmat)$.

Let $\symb{\scS(\R^{d},\fg)}$ denote the space of functions of rapid
decrease from $\R^{d}$ to $\fg$, and $\symb{\left(\scS(\R^{d},\fg)\right)'}$
denote its dual space, consisting of the continuous linear functionals
from $\scS(\R^{d},\fg)$ to $\R$. This is discriminated from $\symb{\scS'(\R^{d},\fg)}$,
the space of $\fg$-valued tempered distributions, which are continuous
linear functionals from $\scS(\R^{d})=\scS(\R^{d},\R)$ to $\fg$.
However, for $F\in\left(\scS(\R^{d},\fg)\right)'$, we can naturally
define the corresponding $\fg$-valued distribution $F^{*}\in\scS'(\R^{2},\fg)$
by

\[
\left\langle \symb{F^{*}}(f),X\right\rangle _{\fg}=F(Xf),\quad X\in\fg,\ f\in\scS(\R^{d},\R),
\]
or more explicitly,

\[
\symb{F^{*}}(f):=\sum_{k=1}^{\dim\fg}F(\gE_{k}f)\gE_{k},\quad f\in\scS(\R^{d},\R),
\]
where $\left\{ \gE_{k}:\ k=1,...,\dim\fg\right\} $ is an orthonormal
basis of $\fg$. So we can identify $\left(\scS(\R^{d},\fg)\right)'$
with $\scS'(\R^{d},\fg)$ under some abuse of notation: If $F\in\left(\scS(\R^{d},\fg)\right)'$
and $f\in\scS(\R^{d},\R)$, let $F(f):=F^{*}(f)\in\fg$. Conversely,
if $F^{*}\in\scS'(\R^{d},\fg)$ and $f\in\scS(\R^{d},\fg)$, let $F^{*}(f):=F(f)\in\R$.

Let $\symb{(\samples,\Prob)}$ be a probability space. Let $\symb W$
be a $\fg$-valued white noise on $\R^{2}$, that is, an isometry
from $L^{2}(\R^{2})$ to $L^{2}((\samples,\Prob),\fg)$. For the same
reason as above, $W$ can also be viewed as an isometry from $L^{2}(\R^{2},\fg)$
to $L^{2}((\samples,\Prob),\R)$. If we consider $W:L^{2}(\R^{2},\fg)\to L^{2}((\samples,\Prob),\R)$,
its covariance is expressed as 
\[
\Ex\left(W(f)W(g)\right)=\left\langle f,g\right\rangle _{L^{2}(\R,\fg)},\quad f,g\in L^{2}(\R^{2},\fg),
\]
and if we consider $W:L^{2}(\R^{2})\to L^{2}((\samples,\Prob),\fg)$,
its covariance is expressed as 
\[
\Ex\bigl(\left\langle W(f),W(g)\right\rangle _{\fg}\bigr)=\left\langle f,g\right\rangle _{L^{2}(\R^{2})},\quad f,g\in L^{2}(\R^{2}),
\]
or more explicitly, 

\[
\Ex\left(W(f)_{k}W(g)_{l}\right)=\delta_{kl}\left\langle f,g\right\rangle _{L^{2}(\R^{2})},\quad f,g\in L^{2}(\R^{2}),\ k,l=1,...,\dim\fg,
\]
where $W(f)_{k}:=\left\langle W(f),\gE_{k}\right\rangle _{\fg}$.
While these views are compatible, we mainly regard $W$ as $W:L^{2}(\R^{2})\to L^{2}((\samples,\Prob),\fg)$
in this paper. 

In the following we write $\symb{L^{p}(\Prob)}:=L^{p}((\samples,\Prob),\R)$
and $\symb{L^{p}(\Prob,\fg)}:=L^{p}((\samples,\Prob),\fg)$.

$W$ is continuous on $\scS(\R^{2})$ a.s., that is, 

\[
\Prob\left[\left(W\upharpoonright\scS(\R^{2})\right)\in\scS'(\R^{2},\fg)\right]=1.
\]
In the following we assume $\left(W_{\omega}\upharpoonright\scS(\R^{2})\right)\in\scS'(\R^{2},\fg)$
for \emph{all $\omega\in\samples,$ }and we simply write this as $W\in\scS'(\R^{2},\fg)$.

Define the $j$th \term{smooth approximation} $W^{(j)}\in C^{\infty}(\R^{2},\fg)$
of $W$ by

\begin{equation}
\symb{W^{(j)}}:=\bS_{j}W.\label{eq:def:W^(j)}
\end{equation}
$W^{(j)}$ converges to $W$ in $\scS'(\R^{2},\fg)$.

\section{Classical gauge theory on $\protect\R^{2}$}

Let $\symb{\Crv=\Crv_{[0,1]}}$ the set of smooth maps $\R\ni t\mapsto\crv(t)=(\crv_{1}(t),\crv_{2}(t))\in\R^{2}$
such that $\supp\dot{\crv}\subset[0,1]$ where $\dot{\crv}(t):=\frac{d}{dt}c(t)$,
in other words, $\crv$ is constant on $(-\infty,0]$ and $[1,\infty)$,
respectively.

For $\crv\in\Crv$, define $\ol{\crv}\in\Crv$ by $\ol{\crv}(t):=\crv(1-t)$.
If two curves$\crv^{(1)},\crv^{(2)}\in\Crv$ satisfy $\crv^{(1)}(1)=\crv^{(2)}(0)$,
we define the \term{concatenation} $\symb{\crv^{(2)}\crv^{(1)}}\in\Crv$
by

\[
\crv^{(2)}\crv^{(1)}(t):=\begin{cases}
\crv^{(1)}(2t) & (t\in(-\infty,1/2])\\
\crv^{(2)}(2t-1) & (t\in[1/2,\infty))
\end{cases},
\]
equivalently, $\crv^{(2)}\crv^{(1)}(t):=\crv^{(1)}(2t)+\crv^{(2)}(2t-1)-\crv^{(2)}(0).$

Fix $\crv\in\Crv_{[0,1]}$. Additionally we assume that any $\crv\in\Crv$
satisfies $\crv_{1}(t)>0$ for all $t$; this assumption is not essential,
but this simplifies the calculations. 

Let $\symb{\Om^{1}=\Om^{1}(\R^{2},\fg)}$ denote the space of $\fg$-valued
smooth 1-forms on $\R^{2}$. An element $A\in\Omega^{1}$ is called
a \term{gauge field} in the physical context. Let $A=A_{1}dx_{1}+A_{1}dx_{2}\in\Om^{1}$
$(A_{1},A_{2}\in C^{\infty}(\R^{2},\fg))$. In the notation $A\left(\dot{\crv}(t)\right)$,
$\dot{\crv}(t)$ should be seen as a tangent vector in the tangent
bundle $T_{\crv(t)}\R^{2}$; that is,
\[
A\left(\dot{\crv}(t)\right)=A\left(\sum_{k=1}^{2}\dot{\crv}_{k}(t)\frac{\di}{\di x^{k}}\right)=\sum_{k=1}^{2}A_{k}(\crv(t))\dot{\crv}_{k}(t).\ 
\]
The \term{parallel transport} $\symb{\ptrans_{\crv,A}(t)}\in G$
$(t\in\R)$ along $\crv\in\Crv_{[0,1]}$ is defined by the ODE

\begin{equation}
\frac{d\ptrans_{\crv,A}(t)}{dt}=A\left(\dot{\crv}(t)\right)\ptrans_{\crv,A}(t)=\sum_{k=1}^{2}A_{k}(\crv(t))\dot{\crv}_{k}(t)\ptrans_{\crv,A}(t),\qquad\ptrans_{\crv,A}(0)=e\label{eq:def:hcA}
\end{equation}
For $t\ge0$, define $X_{t}=X(t)$ to be the line integral of $A$
along $\crv\upha[0,t]$:

\begin{equation}
\symb{X(t)=X_{\crv,A}(t)}:=\int_{\crv\upha[0,t]}A=\int_{0}^{t}A\left(\dot{\crv}(s)\right)ds=\int_{0}^{t}\sum_{k=1}^{2}A_{k}(\crv(s))\dot{\crv}_{k}(s)ds.\label{eq:def:X_A}
\end{equation}
Let $\cV:{\rm Mat}\to L({\rm Mat},{\rm Mat})$ be a bounded smooth
map such that

\begin{equation}
\cV(U)M=MU,\qquad\forall U\in G,\ \forall M\in{\rm Mat}.\label{eq:def:cV}
\end{equation}
(Recall $G:=SU(\nmat)\subset\Mat$.) Then the ODE (\ref{eq:def:hcA})
is rewritten as a normal form 
\begin{equation}
d\ptrans_{\crv}(t)=\cV(\ptrans_{\crv}(t))dX_{\crv,A}(t).\ \label{eq:ptrans-normalform}
\end{equation}

If $\crv$ is a loop (i.e. $\crv(0)=\crv(1)$) , we call $\ptrans_{\crv,A}(1)\in G$
the \term{holonomy} along $\crv$. It is also called the \term{Wilson loop},
mainly when $\ptrans_{\crv,A}(1)$ is a $G$-valued random variable.

The most basic class of loops is that of the \term{simple} (Jordan)
loops, i.e. loops $\crv$ such that if $s,t\in[0,1)$ and $\crv(s)=\crv(t)$
then $s=t$. However, it is useful to consider a slightly broader
class of loops, called lassos (\cite{Dri89,Sen93}).

Let $D\subset\R^{2}$. Suppose $\crv\in\Crv$, $\crv(0)=\crv(1)$,
$\crv$ is simple Let $D\subset\R^{2}$ be the closed domain enclosed
by the arc $\crv([0,1])$. $\crv$ is called a \term{lasso} based
on $x\in\R^{2}$ if there exists $\crv^{1},\crv^{2}\in\Crv$ such
that $\crv^{2}$ is a simple closed curve enclosing $D\subset\R^{2}$
anticlockwise, and that

\[
\crv^{1}(0)=x,\quad\crv^{1}(1)=\crv^{2}(0)=\crv^{2}(1),\quad\crv=\ol{\crv^{1}}\crv^{2}\crv^{1}
\]
In this case, we write

\[
\symb{D(\crv)}:=D,\quad\symb{\totte(\crv)}:=\crv^{1}
\]
A simple loop is also a lasso where $\crv^{1}$ is trivial (i.e. a
constant map). The set of lassos based on $x\in\R^{2}$ is denoted
by $\symb{\Lasso(x)}$, and let $\symb{\Lasso}:=\cup_{x\in\R^{2}}\Lasso(x)$.

Let $\symb{\fD}$ be the set of subsets $D\subset\R^{2}$ such that
there exists a simple loop $\crv\in\Crv$ enclosing $D$. 
\begin{lem}
\label{thm:ptrans-finite-lassos-prod}Fix $A\in\Omega^{1}$. Let $\crv\in\Crv\cap\Lasso(x)$.
Suppose $D_{1},...,D_{n}\in\fD$ satisfy (i) $D(\crv)=\bigcup_{k=1}^{n}D_{k}$,
(ii) $D_{k}^{\circ}\cap D_{l}^{\circ}=\emptyset$ if $k\neq l$, and
(iii) $\left(\bigcup_{1\le l\le k}D_{l}\right)^{\circ}$ is connected
for all $k=1,...,n$. Then there exists $\crv^{1},...,\crv^{n}\in\Crv\cap\Lasso(x)$
such that $D(\crv^{k})=D_{k},\ k=1,...,n$, and
\[
\ptrans_{\crv,A}(1)=\ptrans_{\crv^{n},A}(1)\cdots\ptrans_{\crv^{1},A}(1),\quad
\]
\end{lem}
\begin{proof}
Easily shown by induction for $n$, using the relation $\ptrans_{\ol{\crv^{k}}}=\ptrans_{\crv^{k}}^{-1}$.
\end{proof}
From the definition of $\ptrans$, one can easily show the following:
\begin{lem}
\label{thm:ptrans-F}Fix $x=(x_{1},x_{2})\in\R^{2}$, and suppose
that for each $\ep_{1},\ep_{2}>0$, $\crv_{\ep_{1},\ep_{2}}$ is a
lasso in $\Crv\cap\Lasso$ such that
\[
\crv_{\ep_{1},\ep_{2}}(0)=\crv_{\ep_{1},\ep_{2}}(1)=x,\quad D(\crv_{\ep_{1},\ep_{2}})=[x_{1},x_{1}+\ep_{1}]\times[x_{2},x_{2}+\ep_{2}]\ 
\]
Then

\[
\lim_{\ep_{1},\ep_{2}\to0}\frac{\ptrans_{\crv_{\ep_{1},\ep_{2}},A}(1)-1}{\ep_{1}\ep_{2}}=F_{12}(x),
\]
where $\symb{F_{12}(x)}:=\di_{1}A_{2}(x)-\di_{2}A_{1}(x)+A_{2}(x)A_{1}(x)-A_{1}(x)A_{2}(x).$
\end{lem}
The above $F_{12}=F_{12;A}\in C^{\infty}(\R^{2},\fg)$ is called the
\term{field strength} in physical terminology. The \term{curvature 2-form}
$\symb{F=F_{A}}\in\Omega^{2}(\R^{2},\fg)$ is defined by

\[
\symb{F(x)}:=F_{12}(x)dx_{1}\wedge dx_{2}.
\]
We see $F_{A}=dA+[A,A]$, more exactly, 
\begin{equation}
\symb{F_{A}(X,Y)}=dA(X,Y)+[A(X),A(Y)],\quad X,Y\in T_{x}\R^{2}.\label{eq:def:F}
\end{equation}
However, in this paper we shall impose the axial gauge condition later,
which implies $[A,A]=0$. In this case the linear relation $F=dA$
holds. 

\begin{comment}
Note the following properties:

\begin{shaded}%
\begin{lem}
\label{thm:d_HS-SU-basic}Let $\symb{d_{\HS}(X,Y)}:=\left\Vert X-Y\right\Vert _{\HS}$.
Then

(i) $d_{\HS}(VU_{1},VU_{2})=d_{\HS}(U_{1}V,U_{2}V)=d_{\HS}(U_{1},U_{2})$
for all $U_{1},U_{2},V,\in SU(\nmat),$

(ii) $d_{\HS}(U_{1}\cdots U_{n},I)\le\sum_{k=1}^{n}d_{\HS}(U_{k},I)$
for all $U_{1},...,U_{n}\in SU(\nmat)$.\end{lem}
\end{shaded}
\begin{proof}
\begin{lyxgreyedout}
\hidep{
\begin{proof}
(i)
\[
d(VU_{1},VU_{2})=\left\Vert VU_{1}-VU_{2}\right\Vert _{\HS}=\left[\Tr\left(\left(VU_{1}-VU_{2}\right)^{*}(VU_{1}-VU_{2})\right)\right]^{1/2}=\left[\Tr\left((U_{1}^{*}-U_{2}^{*})V^{*}V(U_{1}-U_{2})\right)\right]^{1/2}\ 
\]

\[
=\left[\Tr\left((U_{1}^{*}-U_{2}^{*})(U_{1}-U_{2})\right)\right]^{1/2}=\left\Vert U_{1}-U_{2}\right\Vert _{\HS}=d(U_{1},U_{2})
\]

\[
d_{\HS}(U_{1}V,U_{2}V)=\left\Vert U_{1}V-U_{2}V\right\Vert _{\HS}=\left[\Tr\left(\left(U_{1}V-U_{2}V\right)^{*}(U_{1}V-U_{2}V)\right)\right]^{1/2}=\left[\Tr\left(V^{*}\left((U_{1}^{*}-U_{2}^{*})\right)^{*}(U_{1}-U_{2})V\right)\right]^{1/2}\ 
\]

\[
=\left[\Tr\left(\left((U_{1}^{*}-U_{2}^{*})\right)^{*}(U_{1}-U_{2})\right)\right]^{1/2}=d(U_{1},U_{2})
\]

(ii) Induction:

\[
d(U_{1}\cdots U_{n+1},I)=d(U_{1}\cdots U_{n},U_{n+1}^{-1})\le d(U_{1}\cdots U_{n},I)+d(I,U_{n+1}^{-1})\le\sum_{k=1}^{n+1}d(U_{k},I)
\]

\end{proof}
}
\end{lyxgreyedout}
\end{proof}
\end{comment}

\section{Axial gauge}

\label{sec:axial-gauge}

For $u\in C^{\infty}(\R^{2},G)$, define the action $\gtrans_{u}$,
called the \term{gauge transformation}, on $A$ by

\[
\symb{\gtrans_{u}A_{k}(x)=A_{k}^{u}(x)}:=u^{-1}(x)A_{k}(x)u(x)-(\di_{k}u^{-1}(x))u(x),
\]
so that
\[
\ptrans_{\crv,\gtrans_{u}A}(t)=u(\crv(t))^{-1}\ptrans_{\crv,A}(t)u(\crv(0)).
\]
Note that if $\crv(0)=\crv(1)$, the holonomies $\ptrans_{\crv,A}(1)$
and $\ptrans_{\crv,\gtrans_{u}A}(1)$ are conjugate. Since

\[
F_{\gtrans_{u}A}(x)=u^{-1}(x)F(x)u(x),
\]
naturally we define the gauge transform of $F$ by $\symb{\gtrans_{u}F}=F^{u}:=u^{-1}Fu$.

Let $\symb{\axial_{\theta}=(e_{\theta1},e_{\theta2})}:=(\cos\theta,\sin\theta)\in\R^{2}\setminus\{0\}$
and $\symb{\axial_{\theta}'=(e_{\theta1}',e_{\theta2}')}:=\axial_{\theta+\pi/2}$.
If $A=A_{1}dx_{1}+A_{2}dx_{2}\in\Omega^{1}$ satisfies $\sum_{k=1}^{2}A_{k}e_{\theta k}\equiv0$
for some $\theta\in[0,2\pi)$, then $A$ is said to be in ($\theta$-)\term{axial gauge}.
In this case we have $[A,A]=0$, and hence $F=dA$. This axial gauge
fixing condition is not complete in that for a given $F=F_{12}dx_{1}\wedge dx_{2}\in\Omega^{2}(\R^{2},\fg)$,
the 1-form $A\in\Omega^{1}$ in $\theta$-axial gauge satisfying $F=dA$
is not unique. Instead if we assume two conditions

\begin{equation}
\sum_{k=1}^{2}A_{k}(x)e_{\theta k}\equiv0,\quad\sum_{k=1}^{2}A_{k}(re_{\theta}')e_{\theta k}'=0,\ \forall r\in\R\ \label{eq:def:theta-gauge}
\end{equation}
we have a unique $A$ for any $F$. In this paper we say that $A$
is in $\theta$\term{-gauge} if these conditions are satisfied. We
see that any $A\in\Omega^{1}$ can be gauge-transformed to satisfy
this condition. If $\theta=0$, $A$ in $\theta$-gauge is determined
by $F$ as follows: 

\begin{equation}
A_{1}(x)\equiv0,\quad A_{2}(x):=\int_{0}^{x_{1}}F_{12}(\xi,x_{2})d\xi,\quad x=(x_{1},x_{2})\in\R^{2}\label{eq:axial-gauge-fix}
\end{equation}
\begin{comment}
We assume (\ref{eq:axial-gauge-fix}) until Sec. \ref{sec:wilson-loop}.
\end{comment}
We assume (\ref{eq:axial-gauge-fix}) in the following. We see

\[
A\left(\dot{\crv}(t)\right)=\int_{0}^{\crv_{1}(t)}F_{12}(x_{1},\crv_{2}(t))\dot{\crv}_{2}(t)dx_{1}.\ \ 
\]
and
\begin{align}
X_{\crv,t} & \equiv X_{\crv}(t)=\int_{0}^{t}A\left(\dot{\crv}(t')\right)dt'=\int_{0}^{t}A_{2}(\crv(t'))\dot{\crv}_{2}(t')dt'\ \\
 & =\int_{0}^{t}\int_{0}^{\crv_{1}(t')}F_{12}(x_{1},\crv_{2}(t'))\dot{\crv}_{2}(t')dx_{1}\, dt'\label{eq:X--W-relation-axial}
\end{align}
Let $\symb{\rectlikedomain_{1}}$ be the set of $E\in\fD$ such that
$E$ is convex w.r.t. $x_{1}$, i.e. 

\begin{equation}
\symb{\rectlikedomain_{1}}:=\left\{ E\in\fD:\ \text{if }(x_{1},x_{2}),(x_{1}',x_{2}),\in E\text{ and }x_{1}\le x_{1}''\le x_{1}',\text{ then }(x_{1}'',x_{2})\in E\right\} .\ \label{eq:def:rectlike}
\end{equation}
Fix $D\in\rectlikedomain_{1}$. Let
\begin{align*}
 & a:=\inf\{x_{2}\in\R:\ \exists x_{1}\in\R,(x_{1},x_{2})\in D\},\ \ \\
 & b:=\sup\{x_{2}\in\R:\ \exists x_{1}\in\R,(x_{1},x_{2})\in D\}.
\end{align*}
Then there exists $\crv^{1},\crv^{2}\in\Crv\cap\Lasso$ such that
$D(\ol{\crv^{2}}\crv^{1})=D$, and that

\[
\crv^{1}(0)=\crv^{2}(0)=a,\quad\crv^{1}(1)=\crv^{2}(1)=b,\quad\crv_{2}^{1}(t)=\crv_{2}^{2}(t),\ \forall t\in[0,1].
\]
Then corresponding parallel transport$\symb{\ptrans_{\crv^{i}}}$
is defined by (\ref{eq:ptrans-normalform}):

\[
d\ptrans_{\crv^{i}}(t)=\cV(\ptrans_{\crv^{i}}(t))dX_{\crv^{i},t},\quad\ptrans_{\crv^{i}}(0)=I.
\]
For $\tau\in[a,b]$, let

\[
\symb{D_{\tau}}:=D\cap(\R\times[a,\crv_{2}^{1}(\tau)]),\quad\symb{F_{\tau}^{D}}:=\int_{D_{\tau}}F_{12}(x)dx.\ 
\]
Let $\crv^{\tau}\in\Crv\cap\Lasso$ satisfy $D(\crv)=D_{\tau}$ and
$\crv_{2}^{\tau}(0)=\crv_{2}^{\tau}(1)=a$. Let $\symb{U(\tau)}:=\ptrans_{\crv^{\tau}}(1)$,
the holonomy of $\crv^{\tau}$.

The following lemmas are easily shown from these definitions:
\begin{lem}
For $t\in[a,b]$, $U(t)=\ptrans_{\crv^{1}}(t)^{-1}\ptrans_{\crv^{2}}(t)$
holds.
\end{lem}
{}
\begin{lem}
For $t\in[a,b]$, 
\[
{\displaystyle U(t)^{-1}\frac{d}{dt}U(t)=-\ptrans_{\crv^{1}}(t)^{-1}\Bigl(\int_{\crv_{1}^{1}(t)}^{\crv_{1}^{2}(t)}F_{12}(x_{1},\crv_{2}^{1}(t))dx_{1}\Bigr)\ptrans_{\crv^{1}}(t)}
\]
holds. Equivalently,
\begin{align}
dU(t) & =-U(t)\ptrans_{\crv^{1}}(t)^{-1}dF_{t}^{D}\ptrans_{\crv^{1}}(t)=-U(t)dB_{t}^{D},\quad\label{eq:dU=00003D-UdB}
\end{align}
where
\[
\symb{B_{t}^{D}}:=\int_{a}^{t}\ptrans_{\crv^{1}}(s)^{-1}dF_{s}^{D}\ptrans_{\crv^{1}}(s).
\]

\end{lem}

\section{operator $\protect\cE$}

Set $F_{12}:=W^{(j)}$, $j$th approximation of the $\fg$-valued
white noise $W$ on $\R^{2}$ defined by (\ref{eq:def:W^(j)}), then
a unique $\Omega^{1}$-valued random variable $\symb{A^{(j)}}$ is
determined by (\ref{eq:axial-gauge-fix}). Let $\symb{X^{(j)}=X_{\crv}^{(j)}}=X_{\crv,A^{(j)}}$,
i.e.
\begin{equation}
\symb{X^{(j)}(t)=X_{\crv,A}^{(j)}(t)}:=\int_{\crv\upha[0,t]}A^{(j)}.\label{eq:def:X^(j)}
\end{equation}
For $H:\R^{2}\to\R$ and $h\in L^{\infty}(\R)$, let

\[
\symb{\hat{\cE}_{\crv}(H,h)}:=\int_{\R}\int_{0}^{\crv_{1}\left(t\right)}\ H(x_{1},\crv_{2}(t))h(t)\dot{\crv}_{2}(t)dx_{1}dt.\ \ 
\]
if the integral in the r.h.s. exists. Let 

\[
\symb{\bigl\Vert\hat{\cE}_{\crv}\bigr\Vert_{2,h}}:=\sup\left\{ \bigl|\hat{\cE}_{\crv}(H,h)\bigr|;H\in L^{2}(\R^{2}),\ \left\Vert H\right\Vert _{L^{2}(\R^{2})}\le1\right\} .\ 
\]
We shall see in Lemma \ref{thm:cE_ch-well-def-1} that $\bigl\Vert\hat{\cE}_{\crv}\bigr\Vert_{2,h}<\infty$
for all $h\in L^{\infty}(\R)$, and hence we can define the bounded
linear operator $\cE_{\crv}:L^{\infty}(\R)\to L^{2}(\R^{2})$ as follows:

\[
\left\langle H,\symb{\cE_{\crv}}h\right\rangle _{L^{2}(\R^{2})}=\hat{\cE}_{\crv}(H,h),\quad H\in L^{2}(\R^{2}),\ h\in L^{\infty}(\R).
\]
Clearly $\supp\left(\cE_{\crv}h\right)\subset\R^{2}$ is compact.
$\symb{W^{(j)}(\cE_{\crv}h)}\in\fg$ is naturally defined by 

\[
\symb{W^{(j)}(\cE_{\crv}h)=\left\langle W^{(j)},\cE_{\crv}h\right\rangle }:=\int_{\R^{2}}W^{(j)}(x)\cdot(\cE_{\crv}h)(x)dx.\ 
\]
This integral is well-defined because $W^{(j)}$ is smooth, $\cE_{\crv}h\in L^{2}(\R^{2})$,
and $\supp\cE_{\crv}h$ is compact. We see the following relations:

\[
W^{(j)}(\cE_{\crv}h)=W(\bS_{j}\cE_{\crv}h)=\hat{\cE}_{\crv}(W^{(j)},h).\ 
\]
We also see 
\begin{equation}
X^{(j)}(t)=\hat{\cE}_{\crv}(W^{(j)},\bOne_{[0,t]})=\left\langle W^{(j)},\cE_{\crv}\bOne_{[0,t]}\right\rangle .\label{eq:X^(j)=00003DcE(W,1)}
\end{equation}
Here define the $\fg$-valued random variable $X(t)$ by

\begin{equation}
\symb{X(t)=X_{\crv}(t)}=W\left(\cE_{\crv}\bOne_{[0,t]}\right)=:\left\langle W,\cE_{\crv}\bOne_{[0,t]}\right\rangle .\label{eq:def:X}
\end{equation}
while the last expression is useful but rather formal because it is
neither a $L^{2}$ inner product, nor a pairing of $\scS'$ and $\scS$.

Hereafter we use the notations such as

\[
\symb{\R_{<}^{2}}:=\left\{ (s,t)\in\R^{2}:\ s<t\right\} ,\quad\symb{[0,T]_{<}^{2}}:=\left\{ (s,t)\in[0,T]^{2}:\ s<t\right\} ,\ \text{etc.}
\]
Let

\[
\symb{T_{\pm}=T_{\crv,\pm}}:=\left\{ t\in(0,1);\ \dot{\crv}_{2}(t)\gtrless0\right\} \quad T_{0}:=\left\{ t\in(0,1);\ \dot{\crv}_{2}(t)=0\right\} 
\]
then these are unions of countable disjoint open intervals: 

\[
T_{\pm}=\bigcup_{i=1}^{N_{\pm}}\symb{I_{\pm,i}},\quad I_{\pm,i}=(\symb{t_{i,0}^{\pm},t_{i,1}^{\pm}}),\quad T_{0}=\bigcup_{i=1}^{N_{0}}I_{0,i},\quad N_{\pm},N_{0}\in\N\cup\left\{ \infty\right\} ,\quad
\]
Define $\symb{\cE_{\crv,i}^{\pm}h}\in L^{2}(\R^{2})$ as follows:
for each $x=(x_{1},x_{2})\in\R^{2}$, let 
\[
\symb{\left(\cE_{\crv,i}^{\pm}h\right)(x)}:=\begin{cases}
h(t) & \text{if }\exists t\in I_{\pm,i},\ x_{2}=\crv_{2}(t),\ 0\le x_{1}\le\crv_{1}\left(t\right)\\
0 & \text{otherwise}
\end{cases},
\]

\[
=\begin{cases}
h(\crv_{2}^{-1}(x_{2};I_{\pm,i})) & \text{if }x_{2}\in\crv_{2}(I_{\pm,i}),\ 0\le x_{1}\le\crv_{1}\left(\crv_{2}^{-1}(x_{2};I_{\pm,i})\right)\\
0 & \text{otherwise}
\end{cases}
\]
where $\crv_{2}^{-1}(x_{2};I_{\pm,i})$ is defined to be $t\in I_{\pm,i}$
such that $\crv_{2}(t)=x_{2}$.

If $\cE_{\crv}h\in L^{2}(\R^{2})$, we can check that $\cE_{\crv}h$
is explicitly expressed by

\begin{equation}
\cE_{\crv}h=\sum_{i=1}^{N_{+}}\cE_{\crv,i}^{+}h-\sum_{i=1}^{N_{-}}\cE_{\crv,i}^{-}h\ \label{eq:cE_ch=00003DE_c^+h-E_c^-h-1}
\end{equation}

\begin{lem}
\label{thm:cE_ch-well-def-1}If we define $\cE_{\crv}h$ by (\ref{eq:cE_ch=00003DE_c^+h-E_c^-h-1}),
then $\cE_{\crv}h\in L^{2}(\R^{2})$ for all $h\in L^{\infty}(\R)$
and $\crv\in\Crv$. \end{lem}
\begin{proof}
If $N_{+}<\infty$ or $N_{-}<\infty$, this is clear. Suppose $N_{+}=N_{-}=\infty$,

Since $\dot{\crv}_{2}(t_{i,0}^{\pm})=\dot{\crv}_{2}(t_{i,1}^{\pm})=0$
for all $i$, and $\sum_{i,\pm}\left(t_{i,1}^{\pm}-t_{i,0}^{\pm}\right)<\infty$,
we have
\begin{align*}
\left|\crv_{2}(t_{i,1}^{\pm})-\crv_{2}(t_{i,0}^{\pm})\right| & =\left|\int_{t_{i,0}^{\pm}}^{t_{i,1}^{\pm}}\dot{\crv}_{2}(t)dt\right|\le\int_{t_{i,0}^{\pm}}^{t_{i,1}^{\pm}}\left|\dot{\crv}_{2}(t)\right|dt\ \\
 & <\int_{t_{i,0}^{\pm}}^{t_{i,1}^{\pm}}\left\Vert \ddot{\crv}_{2}\right\Vert _{L^{\infty}}(t-t_{i,0}^{\pm})dt=\frac{1}{2}\left\Vert \ddot{\crv}_{2}\right\Vert _{L^{\infty}}\left(t_{i,1}^{\pm}-t_{i,0}^{\pm}\right)^{2}
\end{align*}
Thus

\[
\left\Vert \cE_{\crv,i}^{\pm}h\right\Vert _{L^{2}(\R^{2})}^{2}=\int_{\crv_{2}(t_{i,0}^{\pm})}^{\crv_{2}(t_{i,1}^{\pm})}dx_{2}\int_{0}^{\crv_{1}\left(\crv_{2,I_{+,i}}^{-1}(x_{2})\right)}dx_{1}\left|h(\crv_{2,I_{+,i}}^{-1}(x_{2}))\right|^{2}\ 
\]

\[
\le\left\Vert h\right\Vert _{L^{\infty}}^{2}\int_{\crv_{2}(t_{i,0}^{\pm})}^{\crv_{2}(t_{i,1}^{\pm})}dx_{2}\int_{0}^{\crv_{1}\left(\crv_{2,I_{+,i}}^{-1}(x_{2})\right)}dx_{1}
\]

\[
\le\left\Vert h\right\Vert _{L^{\infty}}^{2}\left\Vert \crv_{1}\right\Vert _{L^{\infty}}\left|\crv_{2}(t_{i,1}^{\pm})-\crv_{2}(t_{i,0}^{\pm})\right|\ 
\]

\[
<\frac{1}{2}\left\Vert h\right\Vert _{L^{\infty}}^{2}\left\Vert \crv_{1}\right\Vert _{L^{\infty}}\left\Vert \ddot{\crv}_{2}\right\Vert _{L^{\infty}}\left(t_{i,1}^{\pm}-t_{i,0}^{\pm}\right)^{2}
\]
Therefore we have

\begin{align*}
\left\Vert \cE_{\crv}h\right\Vert _{L^{2}(\R^{2})} & \le\sum_{\pm}\sum_{i=1}^{\infty}\left\Vert \cE_{\crv,i}^{+}h\right\Vert \ \\
 & <\sum_{\pm}\sum_{i=1}^{\infty}\left(\frac{1}{2}\left\Vert h\right\Vert _{L^{\infty}}^{2}\left\Vert \crv_{1}\right\Vert _{L^{\infty}}\left\Vert \ddot{\crv}_{2}\right\Vert _{L^{\infty}}\right)^{1/2}\left(t_{i,1}^{\pm}-t_{i,0}^{\pm}\right)<\infty
\end{align*}

\end{proof}

Define subsets $\Crv_{{\rm \Rot}},\ \Crv_{\infty}$ of $\Crv$ by
\begin{align}
 & \symb{\Crv_{{\rm \Rot}}=\Crv_{[0,1],\Rot}}:=\Bigl\{\crv\in\Crv_{[0,1]}:\ \symb{\Rot(\crv)}:=\sup_{(s,t)\in\R_{<}^{2}}\left\Vert \cE_{\crv}\bOne_{[s,t]}\right\Vert _{L^{\infty}}<\infty\Bigr\},\label{eq:def:Crv_Rot}\\
 & \symb{\Crv_{\infty}=\Crv_{[0,1],\infty}}:=\left\{ \crv\in\Crv_{[0,1]}:\ \left\Vert \cE_{\crv}\right\Vert _{\infty\infty}<\infty\right\} .\label{eq:def:Crv_infty}
\end{align}
where
\[
\symb{\left\Vert \cE_{\crv}\right\Vert _{\infty\infty}}:=\sup\left\{ \left\Vert \cE_{\crv}h\right\Vert _{L^{\infty}};h\in L^{\infty}(\R),\ \left\Vert h\right\Vert _{L^{\infty}}\le1\right\} .\ 
\]
Clearly we see $\Crv_{\infty}\subset\Crv_{\Rot}$. Roughly speaking,
a curve $\crv\in\Crv_{[0,1]}$ is in $\Crv_{[0,1],\Rot}$ if $\crv$
does not rotate (clockwise or anti-clockwise) infinitely many times
around any point in $\R^{2}$, and $\Rot(\crv)$ is the maximum rotation
number of $\crv$. 

Note that in our definition of `smooth curve $\crv$,' possibly $\dot{\crv}(t)=0$
holds for some $t\in(0,1)$. Hence possibly the range $\crv(\R)=\crv([0,1])\subset\R^{2}$
is not a smooth curve in the usual sense. For example, we see that
any (finitely) piecewise linear curves are in $\Crv_{\infty}$ (and
$\Crv_{\Rot}$).

By these definitions we easily find the following:
\begin{lem}
\textup{\label{thm:cE1[s,t]-finiteCombi}If $\crv\in\Crv_{[0,1],\Rot}$,
then }$\cE_{\crv}\bOne_{[s,t]}$ is a finite ($\le2\Rot(\crv)$) linear
combination of characteristic functions; There exists disjoint subsets
$D_{k}\subset\R^{2}\ (-\Rot(\crv)\le k\le\Rot(\crv))$ such that 

\[
\cE_{\crv}\bOne_{[s,t]}=\sum_{k\in-\Rot(\crv)}^{\Rot(\crv)}k\bOne_{D_{k}}.\ 
\]

\end{lem}

\section{Rough paths}

For rough path theory, we refer to \cite{FV10b,FH14}.

Let $V$ be a finite-dimensional linear space, where $V=\fg=\mathfrak{su}(\nmat)$
case is our main concern. Let
\[
\symb{T^{(2)}(V)}:=\R\oplus V\oplus(V\otimes V),\quad
\]
equipped with the truncated tensor product $\otimes$, that is, if
$A=(a,b,c)\in T^{(2)}(V)$ and $A'=(a',b',c')\in T^{(2)}(V)$, define
$\symb{A\otimes A'}$ by

\[
\symb{A\otimes A'}:=(aa',ab'+a'b,ac'+a'c+b\otimes b').\ 
\]
Let $\symb{T_{1}^{(2)}(V)}:=\left\{ (1,b,c)\in T^{(2)}(V)\right\} $.
Then naturally $T_{1}^{(2)}(V)$ becomes a Lie group under $\otimes$.
We denote an element of $T_{1}^{(2)}(V)$ as ${\bf x}=(1,{\bf x}^{[1]},{\bf x}^{[2]})$,
or more readably, ${\bf x}=(1,x,\bbm x)$, etc. 

If ${\bf x}:[0,T]\to T_{1}^{(2)}(\fg)$, we write
\[
\symb{{\bf x}_{s,t}}:={\bf x}_{s}^{-1}\otimes{\bf x}_{t}=\left(1,\ x_{s,t},\ \bbm x_{t}-\bbm x_{s}-x_{s}\otimes x_{s,t}\right),\quad\symb{x_{s,t}}:=x_{t}-x_{s},\ s,t\in[0,T]
\]
If $x\in C^{1\hvar}([0,T],V)$, i.e. $x$ is a continuous path of
bounded variation, define the \term{truncated signature} $\symb{\Sn(x)}:[0,T]_{<}^{2}\to T_{1}^{(2)}(\fg)$
by

\[
\symb{\Sn(x)_{s,t}}:=\left(1,\ x_{s,t},\ \int_{s<u_{1}<u_{2}<t}dx_{u_{1}}\otimes dx_{u_{2}}\right)\in T_{1}^{(2)}(V).
\]
Note that if $x$ is smooth,

\[
\int_{s<u_{1}<u_{2}<t}dx_{u_{1}}\otimes dx_{u_{2}}=\int_{s}^{t}x_{s,r}\otimes dx_{r}=\int_{s}^{t}x_{s,r}\otimes\frac{dx_{r}}{dr}dr.
\]
When $x_{0}=0$ (i.e. $x_{0,t}=x_{t}$), the path
\[
t\mapsto\symb{\lift(x)_{t}}:=\Sn(x)_{0,t}=\left(1,\ x_{t},\ \int_{0}^{t}x_{r}\otimes dx_{r}\right)
\]
is called the \term{(step-2) lift} of $x$.
\begin{thm}
{\rm (Chen's relation \cite[Theorem 7.11, p.133]{FV10b})} \label{thm:Chen's-relation} 

For $x\in C^{1\hvar}\big([0,T],V\big)$ and $0\le s<t<u\le T$, we
have 
\[
\Sn(x)_{s,u}=\Sn(x)_{s,t}\otimes\Sn(x)_{t,u}.
\]

\end{thm}
Define the subgroup $\symb{G^{(2)}(V)}$ of $T_{1}^{(2)}(V)$ by
\begin{equation}
\symb{G^{(2)}(V)}:=\left\{ \Sn(x)_{0,1}:x\in C^{1\hvar}([0,1],V)\right\} \ \label{eq:def:G^(2)(V)}
\end{equation}
It is shown that $G^{(2)}(V)$ is expressed more explicitly as follows:

\begin{equation}
G^{(2)}(V)=\left\{ \left(1,\ x,\ \frac{1}{2}x\otimes x+y\otimes z-z\otimes y\right):\ x,y,z\in V\right\} \ 
\end{equation}
$G^{(2)}(V)$ is given the \term{Carnot-Caratheodory metric} $d_{{\rm CC}}$
\cite{FV10b,FH14}. In this paper, the only information needed for
$d_{{\rm CC}}$ is the following:
\[
d_{{\rm CC}}({\bf x},{\bf y})\simeq\left|y-x\right|+\left|\bbm y-\bbm x-x\otimes(y-x)\right|^{1/2},\quad{\bf x},{\bf y}\in G^{(2)}(V),\ 
\]
where $\left|\cdot\right|$ is the usual norm on the linear space
$T^{(2)}(V)$. In particular, $d_{{\rm CC}}({\bf x},o)\simeq\left|x\right|+\left|\bbm x\right|^{1/2},$
where $o:=1_{G^{(2)}(V)}=(1,0,0)\in G^{(2)}(V)$.

Given ${\bf x},{\bf y}\in C([0,T],G^{2}(V))$, we define the \term{homogeneous
H\"older distance} $C([0,T],G^{(2)}(V))$ by
\begin{align}
\symb{d_{\rpHol\hHol}({\bf x},{\bf y})} & \equiv\symb{d_{{\rm CC};\rpHol\hHol;[0,T]}({\bf x},{\bf y})}:=\sup_{0\le s<t\le T}\frac{d_{{\rm CC}}({\bf x}_{s,t},{\bf y}_{s,t})}{\left|t-s\right|^{\rpHol}}\ \label{eq:def:d_{1/p-Hol}-1}
\end{align}
and let
\[
\symb{C^{\rpHol\hHol}([0,1],G^{(2)}(V))}:=\left\{ {\bf x}\in C([0,T],G^{(2)}(V));\ d_{{\rm CC};\rpHol\hHol;[0,T]}({\bf x},o)<\infty\right\} 
\]

\begin{prop}
{\rm \cite[Proposition 8.12, p.174]{FV10b}} Suppose $1/3<\rpHol\le1/2$,
${\bf x}\in C^{\rpHol\hHol}\big([0,T],G^{(2)}\big(V\big)\big)$ and
${\bf x}_{0}=o$. Then there exists a sequence $(x^{(n)})\subset C^{1\hvar}\big([0,T],V\big)$,
$n\in\N$, such that $\lift(x^{(n)})\to{\bf x}$ uniformly as $n\to\infty$,
i.e.

\[
\lim_{n\to\infty}\sup_{t\in[0,T]}d_{{\rm CC}}({\bf x}_{t},\lift(x^{(n)})_{t})=0.
\]

\end{prop}
If $1/3<\rpHol\le1/2$, $C^{\rpHol\hHol}([0,1],G^{(2)}(V))$ is called
the space of \term{weak geometric $\rpHol$-H\"older rough paths}
\cite{FV10b,FH14}.
\begin{thm}
{\rm (Existence and uniqueness of RDE solution; step-2 case of \cite[Theorem 10.14, p.222]{FV10b}
with \cite[Theorem 10.26, p.233]{FV10b})} \label{thm:Existence-of-RDE-sol}%
\\
Let $d,e\in\N$, $\rpHol\in(1/3,1/2]$ , and assume the following:

(i) $\cV:\R^{e}\to L(\R^{d},\R^{e})$ is in ${\rm Lip}^{\gamma}(\R^{e})$,
where $\gamma>1/\rpHol$,

(ii) $(x^{(n)})_{n\in\N}$ is a sequence in $C^{1\hvar}([0,T],\R^{d})$,
such that
\[
\sup_{n}d_{{\rm CC};\rpHol\hHol;[0,T]}\left(\lift(x^{(n)}),o\right)<\infty.
\]

(iii) ${\bf x}\in C^{\rpHol\hHol}([0,T],G^{(2)}(\R^{d}))$ satisfies
\[
\lim_{n\to\infty}d_{{\rm CC};0\hHol;[0,T]}(\lift(x^{(n)}),{\bf x})=0.
\]

(iv) $y_{0}^{(n)}\in\R^{e}$ is a sequence converging to some $y_{0}$. 

(v) $y^{(n)}$ is the solution of the ODE

\[
dy^{(n)}(t)=\cV(y^{(n)}(t))dx^{(n)}(t),\quad y^{(n)}(0)=y_{0}^{(n)}
\]
Then, $y^{(n)}$ converges in uniform topology to a unique limit $y$
in $C([0,T],\R^{d})$, i.e. $\lim_{n\to\infty}\left\Vert y^{(n)}-y\right\Vert _{L^{\infty}([0,T],\R^{d})}=0.$
\end{thm}
In \cite{FV10b}, $y$ in the above theorem is called the \term{solution of the RDE}
(rough differential equation)
\begin{equation}
dy(t)=\cV(y(t))d{\bf x}(t),\quad y(0)=y_{0},\ \label{eq:RDE}
\end{equation}
and written $y=\symb{\pi_{(\cV)}\big(0,y_{0};{\bf x}\big)}$. Then
we have the following stronger result.
\begin{thm}
{\rm (Existence and uniqueness of full RDE solution; step-2 case
of \cite[Theorem 10.36, p.242]{FV10b} with \cite[Theorem 10.38, p.246]{FV10b})}\label{thm:Existence-of-RDE-sol-full}
Let $d,e\in\N$, $\rpHol\in(1/3,1/2]$, and assume (i)-(iii) in Theorem
\ref{thm:Existence-of-RDE-sol}, and that ${\bf y}_{0}^{(n)}=(1,y_{0}^{(n)},\bbm y_{0}^{(n)})\in G^{(2)}(\R^{e})$
is a sequence converging to some ${\bf y}_{0}$. Then, ${\bf y}_{0}^{(n)}\otimes\lift\big(\pi_{(\cV)}(0,y_{0}^{(n)};x_{n})\big)$
converges in uniform topology to a unique limit ${\bf y}$ in $C([0,T],\R^{d})$,
i.e. 
\[
\lim_{n\to\infty}\sup_{t\in[0,T]}d_{{\rm CC}}({\bf y}^{(n)}(t),{\bf y}(t))=0.
\]

\end{thm}
In \cite{FV10b}, ${\bf y}$ in the above theorem is called the \term{solution of the full RDE}
\begin{equation}
d{\bf y}(t)=\cV(y(t))d{\bf x}(t),\quad{\bf y}(0)={\bf y}_{0},\ \label{eq:RDE-full}
\end{equation}
and written ${\bf y}=\symb{\bpi_{(\cV)}\big(0,y_{0};{\bf x}\big)}$.
$\bpi_{(\cV)}$ is called the \term{It\^o--Lyons map}.
\begin{thm}
\label{thm:fullRDE-conti}Suppose $\rpHol'\le\rpHol$ and $R>0$,
and let $\cV:\R^{e}\to L(\R^{d},\R^{e})$ is in ${\rm Lip}^{\gamma}(\R^{e})$,
for $\gamma>1/\rpHol\ge1$, and let 
\begin{align*}
C_{\le R}^{\rpHol\hHol} & =\symb{C_{\le R}^{\rpHol\hHol}([0,T],G^{(2)}(\R^{d}))}\\
 & :=\left\{ {\bf x}\in C([0,T],G^{(2)}(\R^{d}));\ d_{{\rm CC};\rpHol\hHol;[0,1]}({\bf x},o)\le R\right\} .
\end{align*}
Then, the map
\begin{align*}
\R^{e}\times\Big(C_{\le R}^{\rpHol\hHol},d_{{\rm CC},\rpHol'\hHol}\Big) & \to\Big(C^{\rpHol\hHol}\bigl([0,T],G^{(2)}(\R^{e})\bigr),d_{{\rm CC},\rpHol'\hHol}\Big)\\
(y_{0},{\bf x}) & \mapsto\boldsymbol{\pi}_{(\cV)}(0,y_{0};{\bf x})
\end{align*}
is uniformly continuous. \end{thm}
\begin{proof}
Set $p=1/\rpHol$, $p'=1/\rpHol'$ and $\omega(s,t)=|s-t|$ in \cite[Corollary 10.40, p.247]{FV10b}.\end{proof}
\begin{thm}
{\rm ($N=2$ case of \cite[Theorem A.12, p.583]{FV10b})}\label{thm:G2process-Lq-bound}%\\
 \global\long\def\ordera{\gamma_{1}}
\global\long\def\orderb{\gamma_{2}}
\global\long\def\ordera{\mathfrak{a}}
\global\long\def\orderb{\mathfrak{b}}
%\\
Let $0\le\orderb<\ordera$, and $(\bX_{t}:t\in[0,T])$ be a continuous
$G^{(2)}\big(V\big)$-valued process. Then there exists $q_{0}=q_{0}(\ordera,\orderb)$
and $C=C(\ordera,\orderb,T)$ such that the following holds: if

\[
\left\Vert d_{{\rm CC}}(\bX_{s},\bX_{t})\right\Vert _{L^{q}(\Prob)}\le M\left|t-s\right|^{\ordera},\quad\forall s,t\in[0,T]\ 
\]
holds for some for some $q\ge q_{0}$, then we also have 

\[
\left\Vert d_{{\rm CC};\orderb\hHol:[0,T]}\left(\bX,o\right)\right\Vert _{L^{q}(\Prob)}\le CM
\]

\end{thm}
{}
\begin{thm}
{\em (Kolmogorov $L^{q}$ convergence condition for rough paths \cite[Proposition A.15, p.587]{FV10b})}\label{thm:roughPath-Lq-convergence-G2}
Let ${\bf x}^{(n)}=(1,x^{(n)},\bbm x^{(n)})$ $(n\in\N)$ and ${\bf x}^{(\infty)}=$
$(1,x^{(\infty)},$ $\bbm x^{(\infty)})$ be continuous $G^{(2)}(\R^{d})$-valued
processes defined on $[0,T]$. Let $q\in[1,\infty)$ and assume that
\begin{align}
 & \lim_{n\to\infty}\bigl\Vert d_{{\rm CC}}({\bf x}_{t}^{(n)},{\bf x}_{t}^{(\infty)})\bigr\Vert_{L^{q}(\Prob)}=0\quad\forall t\in[0,T],\label{eq:KC-cond-converge}\\
 & \sup_{1\le n\le\infty}\bigl\Vert d_{{\rm CC};\al\hHol:[0,T]}({\bf x}^{(n)},o)\bigr\Vert_{L^{q}(\Prob)}<\infty,\label{eq:KC-cond-bound}
\end{align}
then for $\al'\in(0,\al)$,
\[
\lim_{n\to\infty}\bigl\Vert d_{{\rm CC},\al'\hHol;[0,T]}({\bf x}^{(n)},{\bf x}^{(\infty)})\bigr\Vert_{L^{q}(\Prob)}=0.
\]

\end{thm}
Note that (\ref{eq:KC-cond-converge}) is equivalent to

\[
\lim_{n}\bigl\Vert x_{t}^{(n)}-x_{t}^{(\infty)}\bigr\Vert_{L^{q}(\Prob,\fg)}=\lim_{n}\bigl\Vert\bbm x_{t}^{(n)}-\bbm x_{t}^{(\infty)}\bigr\Vert_{L^{q/2}(\Prob,\fg\otimes\fg)}=0,\quad\forall t\in[0,T],
\]
and (\ref{eq:KC-cond-bound}) is equivalent to

\[
\sup_{n}\Bigl\Vert\bigl\Vert x^{(n)}\bigr\Vert_{\al\hHol;[0,T]}\Bigr\Vert_{L^{q}(\Prob)}<\infty,\quad\sup_{n}\Bigl\Vert\bigl\Vert\bbm x^{(n)}\bigr\Vert_{2\al\hHol;[0,T]}\Bigr\Vert_{L^{q/2}(\Prob)}<\infty.
\]

\section{Estimate for $X_{s,t}^{(j)}$}

Recall the definitions of $X^{(j)}$ and $X$ (Eqs. (\ref{eq:def:X^(j)}),
(\ref{eq:X^(j)=00003DcE(W,1)}), (\ref{eq:def:X})), and set

\begin{equation}
\symb{X_{s,t}}:=X_{t}-X_{s},\quad\symb{X_{s,t}^{(j)}}:=X_{t}^{(j)}-X_{s}^{(j)}.\label{eq:def:X_{s,t}}
\end{equation}
In this section we prove an estimate for $X_{s,t}^{(j)}$ (Prop. \ref{thm:X^(j)-L^p-bdd}).

\begin{lem}
\label{thm:BesovNorm(p,infty)OfRectangle}\thlab{thm:BesovNorm(p,infty)OfRectangle}For
$D\subset\R^{2}$, let $\bOne_{D}:\R^{2}\to\R$ be the characteristic
function of $D$. Let $x_{1},x_{2},y_{1},y_{2}\in\R$, $a:=x_{2}-x_{1}>0$,
$b:=y_{2}-y_{1}>0$, and $f:=\bOne_{[x_{1},x_{2}]\times[y_{1},y_{2}]}$.
Suppose $p\in[1,\infty),\ \besovorder>0$ and $1-\besovorder p>0$
i.e. $\besovorder\in(0,1/p)$. Then 

\[
\left\Vert f\right\Vert _{B_{p,\infty}^{\besovorder}(\R^{2})}\simeq\left\Vert f\right\Vert _{B_{p,\infty}^{\besovorder}(\R^{2})}'\le\left(ab\right)^{1/p}\left(1+4^{1/p}\min\left\{ a,b\right\} ^{-\besovorder}\right).
\]
Especially if $a\le b\wedge1$, 

\[
\left\Vert f\right\Vert _{B_{p,\infty}^{\besovorder}(\R^{2})}\simeq\left\Vert f\right\Vert _{B_{p,\infty}^{\besovorder}(\R^{2})}'\le5a^{1/p-\besovorder}b^{1/p}.
\]
\end{lem}
\begin{proof}
By Lemma \ref{thm:B_(p,infty)-character} and some elementary (but
rather lengthy) calculations.%
\end{proof}
\begin{lem}
\label{thm:B2infty^{1/2}(R^2)cond-leng}\thlab{thm:B2infty\textasciicircum{}\{1/2\}(R\textasciicircum{}2)cond-2}
Let $D\subset\R^{2}$ be a bounded domain s.t. the boundary $\di D$
is a curve with a finite length $\leng(\di D)\in(0,\infty)$. Then
$\bOne_{D}\in B_{2,\infty}^{\besovorder}(\R^{2})\ $ for all $\besovorder\in(0,1/2].$
More precisely,

\begin{equation}
\left|\bOne_{D}\right|_{B_{2,\infty}^{\besovorder}(\R^{2})}'\le\leng(\di D)\diam(D)^{1/2-\besovorder}\ \label{eq:B2infty^{1/2}(R^2)cond-leng1}
\end{equation}
where $\diam(D)$ is the diameter of $D$. Hence there exists $C=C(\besovorder)>0$
such that

\begin{equation}
\left\Vert \bOne_{D}\right\Vert _{B_{2,\infty}^{\besovorder}(\R^{2})}\le C(\besovorder)\left(\diam(D)+\leng(\di D)\diam(D)^{1/2-\besovorder}\right)\ \label{eq:B2infty^{1/2}(R^2)cond-leng2}
\end{equation}
\end{lem}
\begin{proof}
Let $\symb L=\leng(\di D)$ and $\symb{\delta}:=\diam(D)$. Let $\Leb(A)$
denote the Lebesgue measure of $A\subset\R^{2}$. Then

\[
\left\Vert \bOne_{D}(\cdot+x)-\bOne_{D}\right\Vert _{L^{2}}^{2}=\int_{\R^{2}}\left|\bOne_{D}(y+x)-\bOne_{D}(y)\right|^{2}dy\ 
\]

\begin{equation}
=\int_{D\triangle(D+x)}\left|\bOne_{D}(y+x)-\bOne_{D}(y)\right|^{2}dy\le\Leb(D\triangle(D+x))\ \label{eq:B2infty^{1/2}(R^2)cond-leng-p1}
\end{equation}
If $|x|>\delta$, we see ${\rm Leb}(D\triangle(D+x))=2\Leb(D),$ $\Leb(D\cap(D+x))=0$,
and if $|x|\le\delta$, we have

\[
\Leb(D\triangle(D+x))\le\Leb\Bigl(\bigcup_{t\in[0,1]}(\di D+tx)\Bigr)\le L\left|x\right|\ 
\]
Hence we have
\begin{align*}
\left\Vert \bOne_{D}(\cdot+x)-\bOne_{D}\right\Vert _{L^{2}}^{2} & \le2\Leb(D)\quad\text{if }|x|>\delta\\
\left\Vert \bOne_{D}(\cdot+x)-\bOne_{D}\right\Vert _{L^{2}}^{2} & \le L\left|x\right|\qquad\ \text{if }|x|\le\delta.
\end{align*}
Hence if $|x|\le\delta$,
\begin{align*}
 & \sup_{x\in\R^{2},x\neq0,|x|\le\delta}\left\Vert \bOne_{D}(\cdot+x)-\bOne_{D}\right\Vert _{L^{2}(\R^{2})}\left|x\right|^{-\besovorder}\ \\
 & \le\sup_{x\in\R^{2},x\neq0,|x|\le\delta}\left(L\left|x\right|\right)^{1/2}\left|x\right|^{-\besovorder}\\
 & =\left(L\delta^{1-2\besovorder}\right)^{1/2}\ 
\end{align*}
and if $|x|>\delta$,
\begin{align*}
 & \sup_{x\in\R^{2},|x|>\delta}\left\Vert \bOne_{D}(\cdot+x)-\bOne_{D}\right\Vert _{L^{2}(\R^{2})}\left|x\right|^{-\besovorder}\\
 & \le\sup_{x\in\R^{2},|x|>\delta}\left(2\Leb(D)\right)^{1/2}\left|x\right|^{-\besovorder}\\
 & =2^{1/2}\Leb(D)^{1/2}\delta^{-\besovorder}\\
 & \le2^{-1/2}{\rm \pi^{1/2}}\delta^{1-\besovorder}\quad\text{ (using }\Leb(D)\le\pi\left(\delta/2\right)^{2})\\
 & \le2^{-1}{\rm \pi^{1/2}}\delta^{(1/2)-\besovorder}L^{1/2}\quad(\text{using }\delta\le2^{-1}L)\\
 & \le\delta^{(1/2)-\besovorder}L^{1/2}=\left(L\delta^{1-2\besovorder}\right)^{1/2}\ 
\end{align*}
Thus we have (\ref{eq:B2infty^{1/2}(R^2)cond-leng1}). Moreover, from
\[
\left\Vert \bOne_{D}\right\Vert _{L^{2}}\le\Leb(D)^{1/2}\le\pi^{1/2}\left(\delta/2\right)\ 
\]
we have (\ref{eq:B2infty^{1/2}(R^2)cond-leng2}).\end{proof}
\begin{lem}
\label{thm:|cE.chi[s,t]|_{B_{2,infty}}<}\thlab{thm:|cE.chi{[}s,t{]}|\_\{B\_\{2,infty\}\}<}
Let $\crv\in\Crv_{\Rot}$, $\besovorder\in(0,1/2]$, and $0\le s<t\le1$.
Then, $\cE_{\crv}\bOne_{[s,t]}\in B_{2,\infty}^{\besovorder}(\R^{2})$.
Moreover, when $s$ is sufficiently near to $t$,

\[
\left\Vert \cE_{\crv}\bOne_{[s,t]}\right\Vert _{B_{2,\infty}^{\besovorder}}\lesssim(t-s)^{1/2-\besovorder}.\ 
\]
In other words,

\[
\left\Vert \bDelta_{j}\cE_{\crv}\bOne_{[s,t]}\right\Vert _{L^{2}(\R^{2})}\lesssim\left(t-s\right)^{1/2-\besovorder}2^{-\besovorder j}\ 
\]
\end{lem}
\begin{proof}
Suppose $\crv_{2}(s)\le\crv_{2}(t)$. (The case where $\crv_{2}(s)\ge\crv_{2}(t)$
can be considered similarly.) Let

\[
\symb{D_{n}}:=\left\{ x\in\R^{2};\left(\cE_{\crv}\bOne_{[s,t]}\right)(x)=n\right\} \subset\R^{2},\quad n\in\Z\ 
\]
then by Lemma \ref{thm:cE1[s,t]-finiteCombi} we see

\[
\cE_{\crv}\bOne_{[s,t]}=\sum_{n=-\Rot(\crv)}^{\Rot(\crv)}n\bOne_{D_{n}}.\ 
\]
Define the intervals

\[
\symb{I_{i}}:=\Bigl[\inf_{\tau\in[s,t]}\crv_{i}(\tau),\sup_{\tau\in[s,t]}\crv_{i}(\tau)\Bigr]\subset\R,\quad i=1,2\ 
\]
and the rectangles $R_{1},R_{2}$ in $\R^{2}$ by

\[
\symb{R_{1}}:=\Bigl[0,\inf_{\tau\in[s,t]}\crv_{1}(\tau)\Bigr]\times[\crv_{2}(s),\crv_{2}(t)],\quad\symb{R_{2}}:=I_{1}\times I_{2}.\ \ 
\]
Then we can check the following:

\[
\supp\left(\cE_{\crv}\bOne_{[s,t]}\right)\subset R_{1}\cup R_{2},\quad R_{1}\subset D_{1},\quad n\neq1\then D_{n}\subset R_{2}.\ \ 
\]
Suppose $n\neq1$. Then we see

\[
\diam D_{n}\le\diam R_{2}\lesssim t-s.\ 
\]
We also see that $\di D_{n}\ (n\neq1)$ consists of curve segments
of $\crv$ on $[s,t]$, i.e. $\di D_{n}\subset\crv([s,t])\ (\subset\R^{2})$,
and hence we have 

\[
\leng(\di D_{n})\lesssim t-s.
\]
Hence by Lemma \ref{thm:B2infty^{1/2}(R^2)cond-leng} we have

\[
\left\Vert \bOne_{D_{n}}\right\Vert _{B_{2,\infty}^{\besovorder}}\lesssim t-s+(t-s)(t-s)^{1/2-\besovorder}\simeq t-s.\ 
\]
On the other hand we see

\[
\diam(D_{1}\cap R_{2})\lesssim t-s,\ \leng(\di(D_{1}\cap R_{2}))\lesssim t-s
\]
hence again by Lemma \ref{thm:B2infty^{1/2}(R^2)cond-leng} we have

\[
\left\Vert \bOne_{D_{1}\cap R_{2}}\right\Vert _{B_{2,\infty}^{\besovorder}}\lesssim t-s\ 
\]
Thus by Lemma \ref{thm:BesovNorm(p,infty)OfRectangle}, with $a:=\crv_{2}(t)-\crv_{2}(s)\lesssim t-s$,
$b:=\inf_{\tau\in[s,t]}\crv_{1}(\tau)$, we have 

\[
\left\Vert \bOne_{R_{1}}\right\Vert _{B_{2,\infty}^{\besovorder}}\le C_{1}(\besovorder)a^{1/2-\besovorder}b^{1/2}\le C_{1}(\besovorder)a^{1/2-\besovorder}\left(\sup_{\tau\in[0,1]}\crv_{1}(\tau)\right)^{1/2}\le C_{2}(\besovorder,\crv)(t-s)^{1/2-\besovorder}.\ 
\]
when $s\approx t$. Hence, since $D_{1}=R_{1}\cup(D_{1}\cap R_{2})$, 

\[
\left\Vert \bOne_{D_{1}}\right\Vert _{B_{2,\infty}^{\besovorder}}\le\left\Vert \bOne_{R_{1}}\right\Vert _{B_{2,\infty}^{\besovorder}}+\left\Vert \bOne_{D_{1}\cap R_{2}}\right\Vert _{B_{2,\infty}^{\besovorder}}\lesssim\left(t-s\right)^{1/2-\besovorder}+t-s\simeq\left(t-s\right)^{1/2-\besovorder}.
\]
Thus we have

\[
\left\Vert \cE_{\crv}\bOne_{[s,t]}\right\Vert _{B_{2,\infty}^{\besovorder}}\le\sum_{|n|\le\Rot(\crv)}\left\Vert n\bOne_{D_{n}}\right\Vert _{B_{2,\infty}^{\besovorder}}=\left\Vert \bOne_{D_{1}}\right\Vert _{B_{2,\infty}^{\besovorder}}+\sum_{|n|\le\Rot(\crv),n\neq1}|n|\left\Vert \bOne_{D_{n}}\right\Vert _{B_{2,\infty}^{\besovorder}}\ 
\]

\[
\lesssim\left(t-s\right)^{1/2-\besovorder}+\sum_{|n|\le\Rot(\crv),n\neq1}|n|(t-s)\simeq\left(t-s\right)^{1/2-\besovorder}\ 
\]

\end{proof}

Recall the definitions of $X^{(j)}$, $X$ (Eqs. (\ref{eq:def:X^(j)}),
(\ref{eq:X^(j)=00003DcE(W,1)}), (\ref{eq:def:X})), and of $\symb{X_{s,t}}$,
$\symb{X_{s,t}^{(j)}}$ (Eq. (\ref{eq:def:X_{s,t}})).
\begin{prop}
\label{thm:X^(j)-L^p-converge}\thlab{thm:X\textasciicircum{}(j)-L\textasciicircum{}p-converge}
Let $\crv\in\Crv_{\Rot}$ and $\besovorder\in(0,1/2)$. Then when
$s$ and $t$ are sufficiently near,

\[
\Bigl\Vert X_{s,t}-X_{s,t}^{(j)}\Bigr\Vert_{L^{2}(\Prob,\fg)}\lesssim\left(t-s\right)^{1/2}2^{-j\besovorder}\ \text{ i.e. }\ \Ex\Bigl[\bigl|X_{s,t}-X_{s,t}^{(j)}\bigr|_{\fg}^{2}\Bigr]\lesssim\left(t-s\right)2^{-2j\besovorder}\ 
\]
\end{prop}
\begin{proof}
Since $\cE_{\crv}\bOne_{[s,t]}\in B_{2,\infty}^{\besovorder}(\R^{2})$
and $\left\Vert \Delta_{j}u\right\Vert _{L^{p}(\R^{2})}\le2^{-j\besovorder}\|u\|_{B_{p,\infty}^{\besovorder}}$
we obtain from Lemma \ref{thm:|cE.chi[s,t]|_{B_{2,infty}}<}, 
\begin{align*}
\left\Vert X_{s,t}-X_{s,t}^{(j)}\right\Vert _{L^{2}(\Prob,\fg)} & =\left\Vert \left\langle W,\left(I-\bS_{j}\right)\cE_{\crv}\bOne_{[s,t]}\right\rangle \right\Vert _{L^{2}(\Prob)}=\left\Vert \left(I-\bS_{j}\right)\cE_{\crv}\bOne_{[s,t]}\right\Vert _{L^{2}(\R^{2})}\\
 & \lesssim\left\Vert \cE_{\crv}\bOne_{[s,t]}\right\Vert _{B_{2,\infty}^{\besovorder}}2^{-j\besovorder}\lesssim\left(t-s\right)^{1/2}2^{-j\besovorder}\ 
\end{align*}
\end{proof}
\begin{prop}
\label{thm:X^(j)-L^p-bdd}\thlab{thm:X\textasciicircum{}(j)-L\textasciicircum{}p-bdd}Let
$\crv\in\Crv_{\Rot}$ and $q\in[1,\infty)$. Then there exists $C=C(\crv,q)>0$
such that for all $j\ge-1$, $0\le s<t\le1$ and $\besovorder\in(0,1/2]$, 

\begin{equation}
\Bigl\Vert X_{s,t}^{(j)}\Bigr\Vert_{L^{q}(\Prob)}\le C\left(t-s\right)^{1/2-\besovorder}\ \ \label{eq:X^(j)-L^p-bdd}
\end{equation}
\end{prop}
\begin{proof}
Since $X_{s,t}^{(j)}$ is Gaussian, it suffices to show (\ref{eq:X^(j)-L^p-bdd})
only when $q=2$. By Lemma \ref{thm:|cE.chi[s,t]|_{B_{2,infty}}<},
for any $\besovorder\in(0,1/2]$,
\begin{align*}
 & \left\Vert X_{s,t}^{(j)}\right\Vert _{L^{2}(\Prob)}=\left(\Ex\left[\left|X_{s,t}^{(j)}\right|^{2}\right]\right)^{1/2}=\left\Vert \bS_{j}\cE_{\crv}\bOne_{[s,t]}\right\Vert _{L^{2}(\R^{2})}\le\sum_{i=-1}^{j}\left\Vert \bDelta_{i}\cE_{\crv}\bOne_{[s,t]}\right\Vert _{L^{2}(\R^{2})}\ \\
 & \le\sum_{i=-1}^{j}\left\Vert \cE_{\crv}\bOne_{[s,t]}\right\Vert _{B_{2,\infty}^{\besovorder}}2^{-\besovorder i}\le\sum_{i=-1}^{\infty}2^{-\besovorder i}\left\Vert \cE_{\crv}\bOne_{[s,t]}\right\Vert _{B_{2,\infty}^{\besovorder}}\le C_{1}\left\Vert \cE_{\crv}\bOne_{[s,t]}\right\Vert _{B_{2,\infty}^{\besovorder}}\\
 & \le C_{2}\left(t-s\right)^{1/2-\besovorder}
\end{align*}

\end{proof}

\section{Estimate for $\protect\bbX_{s,t}^{j}$}

For $0\le s<t\le1$ and $j\ge-1$, the $\fg\otimes\fg$-valued random
variable $\bbX_{s,t}^{(j)}$ by

\begin{align}
\symb{\bbX_{s,t}^{(j)}=\bbX_{\crv;s,t}^{(j)}}: & =\int_{s}^{t}X_{s,r}^{(j)}\otimes dX_{r}^{(j)}=\int_{s}^{t}X_{s,r}^{(j)}\otimes\dot{X}_{r}^{(j)}dr\ \ \label{eq:def:bbX}\\
 & =\int_{s}^{t}\left\langle W,\bS_{j}\cE_{\crv}\bOne_{[s,r]}\right\rangle \otimes d\left\langle W,\bS_{j}\cE_{\crv}\bOne_{[0,r]}\right\rangle 
\end{align}
so that $\symb{\bX^{(j)}}\equiv\symb{\bX_{\crv}^{(j)}}:=\left(1,X^{(j)},\bbX^{(j)}\right)=\Sn(X^{(j)}):[0,1]_{<}^{2}\to G^{(2)}(\fg)$.
Let $\symb{\bbX_{t}^{(j)}}:=\bbX_{0,t}^{(j)}$, then $\left(1,X_{t}^{(j)},\bbX_{t}^{(j)}\right)=\lift(X^{(j)})_{t}$.

Fix an orthonormal basis $\gE_{k}\ (k=1,...,\dim\fg)$ of $\fg$,
and set

\[
\bbX_{t}^{(j)}=\sum_{k,l}\symb{\bbX_{t}^{(j);k,l}}\gE_{k}\otimes\gE_{l},\quad\bbX_{t}^{(j);k,l}\in\R.
\]
Let
\[
\symb{\kappa_{j}(x)}:=\left\langle \chchi_{j}(\cdot-x),\ \chchi_{j}(\cdot)\right\rangle _{L^{2}(\R^{2})},\ 
\]
then we see $\kappa_{j}(x-y)=\left\langle \chchi_{j}(\cdot-x),\ \chchi_{j}(\cdot-y)\right\rangle _{L^{2}(\R^{2})}$
and the following:
\begin{lem}
\label{thm:kappa-L1}\thlab{thm:kappa-L1}For all $j\ge-1$,
\[
\int_{\R^{2}}\kappa_{j}(x)dx=\left(\int_{\R^{2}}\chchi_{0}(x)dx\right)^{2}\quad\text{and}\quad\left\Vert \kappa_{j}\right\Vert _{L^{1}(\R^{2})}\le\left\Vert \chchi_{0}\right\Vert _{L^{1}(\R^{2})}^{2}.\ 
\]

\end{lem}
\begin{comment}
\begin{proof}
.

\[
\int_{\R^{2}}dx\kappa_{j}(x)=\int_{\R^{2}}dx\left\langle \chchi_{j}(\cdot-x),\ \chchi_{j}\right\rangle =\int_{\R^{2}}dx\int_{\R^{2}}dy\chchi_{j}(y-x)\chchi_{j}(y)=\int_{\R^{2}}dy\int_{\R^{2}}dx\chchi_{j}(y-x)\chchi_{j}(y)\ 
\]

\[
=\int_{\R^{2}}dy\int_{\R^{2}}dx\chchi_{j}(x)\chchi_{j}(y)=\left(\int_{\R^{2}}dx\chchi_{j}(x)\right)^{2}=\left(\int_{\R^{2}}dx\bS_{-1}(x)\right)^{2}\ 
\]

\[
\left\Vert \kappa_{j}\right\Vert _{L^{1}}=\int_{\R^{2}}dx\left|\kappa_{j}(x)\right|=\int_{\R^{2}}dx\left|\left\langle \chchi_{j}(\cdot-x),\ \chchi_{j}\right\rangle \right|=\int_{\R^{2}}dx\left|\int_{\R^{2}}dy\chchi_{j}(y-x)\chchi_{j}(y)\right|\ 
\]

\[
\le\int_{\R^{2}}dx\int_{\R^{2}}dy\left|\chchi_{j}(y-x)\chchi_{j}(y)\right|=\int_{\R^{2}}dy\int_{\R^{2}}dx\left|\chchi_{j}(y-x)\right|\left|\chchi_{j}(y)\right|\ 
\]

\[
=\int_{\R^{2}}dy\int_{\R^{2}}dx\left|\chchi_{j}(x)\right|\left|\chchi_{j}(y)\right|=\left(\int_{\R^{2}}dx\left|\chchi_{j}(x)\right|\right)^{2}=\left\Vert \chchi_{j}\right\Vert _{L^{1}}^{2}=\left\Vert \chchi_{0}\right\Vert _{L^{1}}^{2}\ 
\]
\end{proof}
\end{comment}

Let
\begin{equation}
\symb{f_{t}=f_{\crv,t}^{j}}:=\frac{d}{dt}\bS_{j}\cE_{\crv}\bOne_{[s,t]}\in\scS(\R^{2}).\label{eq:def:f_t}
\end{equation}
We see

\[
f_{\crv,t}^{j}(x_{1},x_{2})=\dot{\crv}_{2}(t)\int_{0}^{\crv_{1}(t)}\chchi_{j}(x_{1}-\xi,x_{2}-\crv_{2}(t))d\xi,\ 
\]
and

\begin{equation}
\left\langle f_{\crv,r_{1}}^{j},f_{\crv,r_{2}}^{j}\right\rangle =\dot{\crv}_{2}(r_{1})\dot{\crv}_{2}(r_{2})\int_{0}^{\crv_{1}(r_{1})}d\xi\int_{0}^{\crv_{1}(r_{2})}d\xi'\kappa_{j}(\xi-\xi',\crv_{2}(r_{1})-\crv_{2}(r_{2})).\ \label{eq:<f_s,f_t>}
\end{equation}

\begin{lem}
\label{thm:|<S_jcEchi_[s,r1],f_{r2}>|<C}\thlab{thm:|<S\_jcEchi\_{[}s,r1{]},f\_\{r2\}>|<C}
For any $\crv\in\Crv_{\Rot}$, there exists $C=C(\crv)>0$ such that
for all $j\ge-1$ and $r_{1},r_{2}\ge s$,

\[
\left|\left\langle \bS_{j}\cE_{\crv}\bOne_{[s,r_{1}]},f_{\crv,r_{2}}^{j}\right\rangle \right|\le C,\ \ 
\]
and hence

\[
\left|\int_{s}^{t}dr\left\langle \bS_{j}\cE_{\crv}\bOne_{[s,r]},f_{\crv,r}^{j}\right\rangle \right|\le(t-s)C\ \ 
\]
\end{lem}
\begin{proof}
Let
\[
D_{[t,t+\ep]}:=\left\{ x\in\R^{2};\ \tau\in[t,t+\ep],\ \crv_{2}(\tau)=x_{2},\ \tau\in[t,t+\ep],\ 0\le x_{1}\le\crv_{1}(\tau)\right\} 
\]
then we see
\[
\Leb\left(D_{[t,t+\ep]}\right)\simeq\crv_{1}(t)\dot{\crv}_{2}(t)\ep
\]
for $\ep\simeq0$. Since $f_{t}=0$ if $\dot{\crv}_{2}(t)=0$, we
suppose $\dot{\crv}_{2}(t)>0$ without loss of generality ($\dot{\crv}_{2}(t)<0$
case is similar). Then we see for sufficiently small $\ep>0$,

\[
\cE_{\crv}\bOne_{[t,t+\ep]}=\bOne_{D_{[t,t+\ep]}}.
\]
Hence, using $\bS_{j}u=\chchi_{j}*u$ and the inequality$\left\Vert \phi*\psi\right\Vert _{L^{q}}\le\left\Vert \phi\right\Vert _{L^{1}}\left\Vert \psi\right\Vert _{L^{q}}\ (q\in[1,\infty])$,
we have
\begin{align*}
\left\Vert f_{t}\right\Vert _{L^{1}} & =\left\Vert \lim_{\ep\to+0}\ep^{-1}\bS_{j}\cE_{\crv}\bOne_{[t,t+\ep]}\right\Vert _{L^{1}}=\lim_{\ep\to+0}\left\Vert \ep^{-1}\bS_{j}\cE_{\crv}\bOne_{[t,t+\ep]}\right\Vert _{L^{1}}\\
 & \le\lim_{\ep\to+0}\ep^{-1}\left\Vert \chchi_{j}\right\Vert _{L^{1}}\left\Vert \cE_{\crv}\bOne_{[t,t+\ep]}\right\Vert _{L^{1}}=\left\Vert \chchi_{j}\right\Vert _{L^{1}}\lim_{\ep\to+0}\ep^{-1}\left\Vert \bOne_{D_{[t,t+\ep]}}\right\Vert _{L^{1}}\\
 & =\left\Vert \chchi_{j}\right\Vert _{L^{1}}\lim_{\ep\to+0}\ep^{-1}\Leb\left(D_{[t,t+\ep]}\right)=\left\Vert \chchi_{0}\right\Vert _{L^{1}}\crv_{1}(t)\dot{\crv}_{2}(t)\\
 & \le C\left\Vert \chchi_{0}\right\Vert _{L^{1}},\quad C:=\sup_{r\in[0,1]}\left|\crv_{1}(r)\right|\sup_{r\in[0,1]}\left|\dot{\crv}_{2}(r)\right|.
\end{align*}
Thus
\begin{align*}
\left|\left\langle \bS_{j}\cE_{\crv}\bOne_{[s,r_{1}]},f_{r_{2}}\right\rangle \right| & \le\left\Vert \bS_{j}\cE_{\crv}\bOne_{[s,r_{1}]}\right\Vert _{L^{\infty}}\left\Vert f_{r_{2}}\right\Vert _{L^{1}}\le C\left\Vert \chchi_{j}\right\Vert _{L^{1}}\left\Vert \cE_{\crv}\bOne_{[s,r_{1}]}\right\Vert _{L^{\infty}}\left\Vert \chchi_{0}\right\Vert _{L^{1}}\\
 & =C\left\Vert \chchi_{0}\right\Vert _{L^{1}}^{2}\left\Vert \cE_{\crv}\bOne_{[s,r_{1}]}\right\Vert _{L^{\infty}}\le C\left\Vert \chchi_{0}\right\Vert _{L^{1}}^{2}\Rot(\crv).
\end{align*}
\end{proof}
\begin{prop}
\label{thm:int_Rdr|<f1,f2>|<} For any $\crv\in\Crv_{\infty}$, there
exists $C=C(\crv)>0$ such that for all $j\ge-1$, and $r_{1}\in[0,1]$,
\[
\int_{0}^{1}\left|\left\langle f_{\crv,r_{1}}^{j},f_{\crv,r_{2}}^{j}\right\rangle \right|dr_{2}\le C.\ 
\]
\end{prop}
\begin{proof}
Let
\[
\symb{H_{\crv,j,r_{1}}(x_{1},x_{2})}=\int_{[0,\crv_{1}(r_{1})]}\left|\kappa_{j}\left((\xi_{1},\crv_{2}(r_{1}))-(x_{1},x_{2})\right)\right|d\xi_{1}.\ 
\]
We easily check $\left\Vert H_{\crv,j,r_{1}}\right\Vert _{L^{1}(\R^{2})}=\crv_{1}(r_{1})\left\Vert \kappa_{j}\right\Vert _{L^{1}(\R^{2})}$,
hence By Prop. \ref{thm:kappa-L1}, we have
\begin{align}
 & \left\Vert H_{\crv,j,r_{1}}\right\Vert _{L^{1}(\R^{2})}\le\crv_{1}(r_{1})\left\Vert \chchi_{0}\right\Vert _{L^{1}}^{2}.\ \label{eq:|H_{jcr}|_{L1}<}
\end{align}
Let $s_{\crv}(t)=\sgn(\dot{\crv}_{2}(t))$, i.e.
\[
\symb{s_{\crv}(t)}:=\dot{\crv}_{2}(t)/\left|\dot{\crv}_{2}(t)\right|,\quad0\le t\le1\ 
\]
where $s_{\crv}(t):=0$ if $\dot{\crv}_{2}(t)=0$. Then by (\ref{eq:<f_s,f_t>})
and (\ref{eq:|H_{jcr}|_{L1}<}) we have,
\begin{align*}
 & \int_{\R}dr_{2}\left|\left\langle f_{\crv,r_{1}}^{j},f_{\crv,r_{2}}^{j}\right\rangle \right|\\
 & =\int_{\R}\left|\dot{\crv}_{2}(r_{1})\dot{\crv}_{2}(r_{2})\int_{0}^{\crv_{1}(r_{1})}\int_{0}^{\crv_{1}(r_{2})}\kappa_{j}\left((\xi_{1},\crv_{2}(r_{1}))-(\xi_{2},\crv_{2}(r_{2}))\right)d\xi_{2}d\xi_{1}\right|dr_{2}\ \\
 & \le\left|\dot{\crv}_{2}(r_{1})\right|\int_{\R}\dot{\crv}_{2}(r_{2})s_{\crv}(r_{2})\\
 & \ \qquad\times\int_{0}^{\crv_{1}(r_{2})}\int_{0}^{\crv_{1}(r_{1})}\left|\kappa_{j}\left((\xi_{1},\crv_{2}(r_{1}))-(\xi_{2},\crv_{2}(r_{2}))\right)\right|d\xi_{1}d\xi_{2}dr_{2}\\
 & =\left|\dot{\crv}_{2}(r_{1})\right|\hat{\cE}_{\crv}\left(H_{\crv,j,r_{1}},s_{\crv}\right)=\left|\dot{\crv}_{2}(r_{1})\right|\left\langle H_{\crv,j,r_{1}},\cE_{\crv}s_{\crv}\right\rangle \ \\
 & \le\left|\dot{\crv}_{2}(r_{1})\right|\left\Vert H_{\crv,j,r_{1}}\right\Vert _{L^{1}}\left\Vert \cE_{\crv}s_{\crv}\right\Vert _{L^{\infty}}\\
 & \le\left|\dot{\crv}_{2}(r_{1})\right|\left\Vert H_{j,\crv,r_{1}}\right\Vert _{L^{1}}\left\Vert \cE_{\crv}\right\Vert _{\infty\infty}\\
 & \le\left|\dot{\crv}_{2}(r_{1})\right|\crv_{1}(r_{1})\left\Vert \chchi_{0}\right\Vert _{L^{1}}^{2}\left\Vert \cE_{\crv}\right\Vert _{\infty\infty}\\
 & \le C(\crv)\ 
\end{align*}
\end{proof}
\begin{lem}
\label{thm:|<ScEchi,ScEchi>|<} For any $\crv\in\Crv_{\Rot}$, there
exists $C=C(\crv)>0$ such that for all $j\ge-1$ and $r_{1},r_{2}\in[s,t]$,
\[
\left|\left\langle \bS_{j}\cE_{\crv}\bOne_{[s,r_{1}]},\bS_{j}\cE_{\crv}\bOne_{[s,r_{2}]}\right\rangle \right|\le C(t-s).\ 
\]
\end{lem}
\begin{proof}
We see
\begin{align*}
\left\Vert \cE_{\crv}\bOne_{[s,r_{2}]}\right\Vert _{L^{1}} & \le\Rot(\crv)\Leb(\supp\cE_{\crv}\bOne_{[s,r_{2}]})\ \\
 & \le\Rot(\crv)\sup_{t_{1},t_{2}\in[s,t]}\left|\crv_{2}(t_{1})-\crv_{2}(t_{2})\right|\sup_{t_{1}\in[s,t]}\crv_{1}(t_{1}).
\end{align*}
Hence, using $\left\Vert f*g\right\Vert _{L^{q}}\le\left\Vert f\right\Vert _{L^{1}}\left\Vert g\right\Vert _{L^{q}}\ (q\in[0,\infty])$
and $\bS_{j}u=\chchi_{j}*u$ we have
\begin{align*}
 & \left|\left\langle \bS_{j}\cE_{\crv}\bOne_{[s,r_{1}]},\bS_{j}\cE_{\crv}\bOne_{[s,r_{2}]}\right\rangle \right|\le\left\Vert \bS_{j}\cE_{\crv}\bOne_{[s,r_{1}]}\right\Vert _{L^{\infty}}\left\Vert \bS_{j}\cE_{\crv}\bOne_{[s,r_{2}]}\right\Vert _{L^{1}}\ \\
 & \le\left\Vert \bS_{j}\cE_{\crv}\bOne_{[s,r_{1}]}\right\Vert _{L^{\infty}}\left\Vert \bS_{j}\cE_{\crv}\bOne_{[s,r_{2}]}\right\Vert _{L^{1}}\ \\
 & \le\left\Vert \bS_{j}\cE_{\crv}\bOne_{[s,r_{1}]}\right\Vert _{L^{\infty}}\left\Vert \chchi_{j}\right\Vert _{L^{1}}\left\Vert \cE_{\crv}\bOne_{[s,r_{2}]}\right\Vert _{L^{1}}\ \\
 & \le\left\Vert \chchi_{j}\right\Vert _{L^{1}}\Rot(\crv)\left\Vert \chchi_{j}\right\Vert _{L^{1}}\left\Vert \cE_{\crv}\bOne_{[s,r_{2}]}\right\Vert _{L^{1}}\\
 & \le\left\Vert \chchi_{0}\right\Vert _{L^{1}}^{2}\Rot(\crv)^{2}\sup_{t_{1},t_{2}\in[s,t]}\left|\crv_{2}(t_{1})-\crv_{2}(t_{2})\right|\sup_{t_{1}\in[s,t]}\crv_{1}(t_{1})\\
 & \le C_{1}\sup_{t_{1},t_{2}\in[s,t]}\left|\crv_{2}(t_{1})-\crv_{2}(t_{2})\right|\\
 & \le C_{1}\sup_{t_{1}\in[s,t]}\left|\dot{\crv}_{2}(t_{1})\right|(t-s)\le C_{2}(t-s).\ 
\end{align*}
\end{proof}
\begin{prop}
\label{thm:Ex[|bbX|^2]<}For any $\crv\in\Crv_{\infty}$ and $p\in[1,\infty)$,
there exists $C=C(\crv,p)>0$ such that for all $j\ge-1$ and $0\le s<t\le1$,

\[
\left\Vert \bbX_{s,t}^{(j)}\right\Vert _{L^{p}(\Prob,\fg)}\le C(t-s).\ 
\]
\end{prop}
\begin{proof}
Since $X^{(j)}$ is Gaussian, all $L^{p}$-norms ($p\in[1,\infty)$)
for $\bbX^{(j)}$ are equivalent by \cite[Theorem 3.50 p.39]{Jan97}.
Hence it is enough to show the bound for $p=2$. Using the equation
\begin{equation}
\Ex[ABCD]=\Ex[AB]\Ex[CD]+\Ex[AC]\Ex[BD]+\Ex[AD]\Ex[BC]\label{eq:ExABCD-gauss}
\end{equation}
for any Gaussian random variables $A,B,C,D$, we have
\begin{align*}
 & \Ex\left[\left|\bbX_{s,t}^{(j);k,l}\right|^{2}\right]=\Ex\left[\left|\int_{s}^{t}dr\left\langle W,\bS_{j}\cE_{\crv}\bOne_{[s,r]}\right\rangle ^{k}\left\langle W,f_{r}\right\rangle ^{l}\right|^{2}\right]\\
 & =\int_{s}^{t}\int_{s}^{t}\Ex\left[\left\langle W,\bS_{j}\cE_{\crv}\bOne_{[s,r_{1}]}\right\rangle ^{k}\left\langle W,f_{r_{1}}\right\rangle ^{l}\left\langle W,\bS_{j}\cE_{\crv}\bOne_{[s,r_{2}]}\right\rangle ^{k}\left\langle W,f_{r_{2}}\right\rangle ^{l}\right]dr_{2}dr_{1}\\
 & =\int_{s}^{t}\int_{s}^{t}\Big(\delta_{kl}\left\langle \bS_{j}\cE_{\crv}\bOne_{[s,r_{1}]},f_{r_{1}}\right\rangle \delta_{kl}\left\langle \bS_{j}\cE_{\crv}\bOne_{[s,r_{2}]},f_{r_{2}}\right\rangle dr_{2}dr_{1}\\
 & \quad+\left\langle \bS_{j}\cE_{\crv}\bOne_{[s,r_{1}]},\bS_{j}\cE_{\crv}\bOne_{[s,r_{2}]}\right\rangle \left\langle f_{r_{1}},f_{r_{2}}\right\rangle \\
 & \quad+\delta_{kl}\left\langle \bS_{j}\cE_{\crv}\bOne_{[s,r_{1}]},f_{r_{2}}\right\rangle \delta_{kl}\left\langle f_{r_{1}},\bS_{j}\cE_{\crv}\bOne_{[s,r_{2}]}\right\rangle \Big)\\
 & =\delta_{kl}\left(\int_{s}^{t}dr\left\langle \bS_{j}\cE_{\crv}\bOne_{[s,r]},f_{r}\right\rangle \right)^{2}\\
 & \quad+\int_{s}^{t}dr_{1}\int_{s}^{t}dr_{2}\left\langle \bS_{j}\cE_{\crv}\bOne_{[s,r_{1}]},\bS_{j}\cE_{\crv}\bOne_{[s,r_{2}]}\right\rangle \left\langle f_{r_{1}},f_{r_{2}}\right\rangle \\
 & \quad+\delta_{kl}\int_{s}^{t}\int_{s}^{t}\left\langle \bS_{j}\cE_{\crv}\bOne_{[s,r_{1}]},f_{r_{2}}\right\rangle \left\langle \bS_{j}\cE_{\crv}\bOne_{[s,r_{2}]},f_{r_{1}}\right\rangle dr_{2}dr_{1}\ \ \\
 & =:\symb{{\rm (I)}+{\rm (II)}+{\rm (III)}}\ 
\end{align*}
By Lemma \ref{thm:|<S_jcEchi_[s,r1],f_{r2}>|<C} we find
\[
{\rm (I)}\le C_{1}\delta_{kl}(t-s)^{2}\ 
\]
By Lemma \ref{thm:|<ScEchi,ScEchi>|<} we find
\begin{align*}
 & {\rm (II)}=\int_{s}^{t}\int_{s}^{t}\left\langle \bS_{j}\cE_{\crv}\bOne_{[s,r_{1}]},\bS_{j}\cE_{\crv}\bOne_{[s,r_{2}]}\right\rangle \left\langle f_{r_{1}},f_{r_{2}}\right\rangle dr_{2}dr_{1}\\
 & \le\int_{s}^{t}\int_{s}^{t}C_{2}(t-s)\left|\left\langle f_{r_{1}},f_{r_{2}}\right\rangle \right|dr_{2}dr_{1}\\
 & \le\int_{s}^{t}C_{2}(t-s)\int_{\R}\left|\left\langle f_{r_{1}},f_{r_{2}}\right\rangle \right|dr_{2}dr_{1}\ 
\end{align*}
Hence by Prop \ref{thm:int_Rdr|<f1,f2>|<}, we have

\[
{\rm (II)}\le C_{3}(t-s)^{2}\ 
\]
By Lemma \ref{thm:|<S_jcEchi_[s,r1],f_{r2}>|<C} we find
\[
{\rm (III)}\le\delta_{kl}\int_{s}^{t}\int_{s}^{t}\ C_{5}dr_{2}dr_{1}=\delta_{kl}C_{4}(t-s)^{2}\ 
\]
Thus we have
\[
\Ex\left[\left|\bbX_{s,t}^{(j);k,l}\right|^{2}\right]={\rm (I)}+{\rm (II)}+{\rm (III)}\le C_{5}(t-s)^{2}.
\]

\end{proof}
Notice the following properties of delta functions:
\begin{lem}
\label{thm:marginal-delta}\thlab{thm:marginal-delta} Let $\delta\in\scS'(\R^{2})$
denote the Dirac delta function, and suppose that $D\subset\R^{2}$
is bounded and measurable. Then

(i) If $0\in D^{\circ}$, $\lim_{j,j'\to\infty}\left\langle \bS_{j}\bOne_{D},\bS_{j'}\delta\right\rangle =1$.

(ii) If $0\in\left(\R^{2}\setminus D\right)^{\circ}$, $\lim_{j,j'\to\infty}\left\langle \bS_{j}\bOne_{D},\bS_{j'}\delta\right\rangle =0$.

(iii) If $0\in\di D$ and $\di D$ is a smooth curve on some neighborhood
of $0$, $\lim_{j,j'\to\infty}\left\langle \bS_{j}\bOne_{D},\bS_{j'}\delta\right\rangle =\frac{1}{2}$.
\end{lem}
{}
\begin{prop}
\label{thm:lim_{jj'}|bbX-bbX|=00003D0}\thlab{thm:lim\_\{jj'\}|bbX-bbX|=0}For
each $\crv\in\Crv_{\infty}$ and $s,t\in[0,1]$, $\bigl(\bbX_{\crv;s,t}^{(j)}\bigr)_{j\ge-1}$
is Cauchy in $L^{p}(\Prob,\fg)$ for any $p\in[1,\infty)$, i.e.
\[
\lim_{j,j'\to\infty}\left\Vert \bbX_{\crv;s,t}^{(j')}-\bbX_{\crv;s,t}^{(j)}\right\Vert _{L^{p}(\Prob,\fg\otimes\fg)}=0.
\]
\end{prop}
\begin{proof}
This result follows immediately from Lemmas \ref{thm:|X^(j+1)-X^(j)|_{L^p}<N1+N2-jj'},
\ref{thm:lim_{j,j'}N_1=00003D0} and \ref{thm:lim_{jj'}N_2=00003D0}
below.\end{proof}
\begin{lem}
\label{thm:|X^(j+1)-X^(j)|_{L^p}<N1+N2-jj'}\thlab{thm:|X\textasciicircum{}(j+1)-X\textasciicircum{}(j)|\_\{L\textasciicircum{}p\}<N1+N2-jj'}
For any $\crv\in\Crv$,
\[
\left\Vert \bbX_{\crv;s,t}^{(j');k,l}-\bbX_{\crv;s,t}^{(j);k,l}\right\Vert _{L^{2}(\Prob)}\le N_{1}+N_{2}\ 
\]
where
\begin{align}
 & \symb{N_{1}}:=\left\Vert \int_{s}^{t}\left\langle W,\bS_{j'}\cE_{\crv}\bOne_{[s,r]}\right\rangle ^{k}\left\langle W,f_{\crv,r}^{j'}-f_{\crv,r}^{j}\right\rangle ^{l}dr\right\Vert _{L^{2}(\Prob)}\label{eq:def:N_1}\\
 & \symb{N_{2}}:=\left\Vert \int_{s}^{t}\left\langle W,f_{\crv,r}^{j}\right\rangle ^{l}\left\langle W,\bS_{j'}\cE_{\crv}\bOne_{[s,r]}-\bS_{j}\cE_{\crv}\bOne_{[s,r]}\right\rangle ^{k}dr\right\Vert _{L^{2}(\Prob)}\label{eq:def:N_2}
\end{align}
\end{lem}
\begin{proof}
By the definition (\ref{eq:def:bbX}) of $\bbX^{(j)}$, we see
\[
\bbX_{s,t}^{(j)}=\int_{s}^{t}\left\langle W,\bS_{j}\cE_{\crv}\bOne_{[s,r]}\right\rangle \otimes\left\langle W,f_{r}^{j}\right\rangle dr,
\]
and hence the bound easily follows from (\ref{eq:ExABCD-gauss}).

\end{proof}
Let 
\begin{equation}
\symb{\delta_{\crv,t}}:=\frac{d}{dt}\cE_{\crv}\bOne_{[0,t]}=\dot{\crv}_{2}(t)\int_{0}^{\crv_{1}(r)}\delta_{(\xi,\crv_{2}(t))}d\xi\in\scS'(\R^{2}).\ \label{eq:def:delta_{c,t}}
\end{equation}
where

\[
\delta_{x}(y):=\delta(y-x),\quad x,y\in\R^{2}.
\]

\begin{equation}
\symb{\bS_{j,j'}}:=\bS_{j'}-\bS_{j}=\sum_{i=j}^{j'-1}\bDelta_{i}.\ \label{eq:def:S_{jj'}}
\end{equation}

\begin{equation}
\symb{\bchi_{j,j'}}:=\bchi_{j'}-\bchi_{j}=\sum_{i=j}^{j'-1}\brho_{i}.\ \label{eq:def:chi_{jj'}}
\end{equation}
We see $f_{t}\equiv f_{\crv,t}^{j}=\bS_{j}\delta_{\crv,t}$.
\begin{lem}
\label{thm:N_1^2=00003DI_1^2+I_2+I_3-jj'}\thlab{thm:N\_1\textasciicircum{}2=I\_1\textasciicircum{}2+I\_2+I\_3-jj'}For
any $\crv\in\Crv$,
\[
N_{1}^{2}=I_{1}^{2}+I_{2}+I_{3}\ 
\]
where
\begin{align}
 & \symb{I_{1}}:=\int_{s}^{t}\delta_{kl}\left\langle \bS_{j'}\cE_{\crv}\bOne_{[s,r]},\ \bS_{j,j'}\delta_{\crv,r}\right\rangle dr\ \label{eq:def:I_1}\\
 & \symb{I_{2}}:=\int_{s}^{t}\int_{s}^{t}\left\langle \bS_{j'}\cE_{\crv}\bOne_{[s,r]},\ \bS_{j'}\cE_{\crv}\bOne_{[s,r']}\right\rangle \left\langle \bS_{j,j'}\delta_{\crv,r},\ \bS_{j,j'}\delta_{\crv,r'}\right\rangle dr'dr\label{eq:def:I_2}\\
 & \symb{I_{3}}:=\delta_{kl}\int_{s}^{t}\int_{s}^{t}\left\langle \bS_{j'}\cE_{\crv}\bOne_{[s,r]},\ \bS_{j,j'}\delta_{\crv,r'}\right\rangle \left\langle \bS_{j,j'}\delta_{\crv,r},\ \bS_{j'}\cE_{\crv}\bOne_{[s,r']}\right\rangle dr'dr\ \label{eq:def:I_3}
\end{align}
\end{lem}
\begin{proof}
.

By a straightforward calculation, using (\ref{eq:ExABCD-gauss}).%
\end{proof}
\begin{lem}
\label{thm:lim_{j,j'}I_1=00003D0}For any $\crv\in\Crv_{\Rot}$,
\begin{equation}
\lim_{j,j',j''\to\infty}\int_{s}^{t}\left\langle \bS_{j''}\cE_{\crv}\bOne_{[s,r]},\bS_{j,j'}\delta_{\crv,r}\right\rangle dr=0.\label{eq:lim_{j,j'}I_1=00003D0}
\end{equation}
Especially, $I_{1}:=\int_{s}^{t}\delta_{kl}\left\langle \bS_{j'}\cE_{\crv}\bOne_{[s,r]},\bS_{j,j'}\delta_{\crv,r}\right\rangle dr$
is Cauchy in $j,j'$, i.e. $\lim_{j,j'\to\infty}I_{1}=0$.\end{lem}
\begin{proof}
By Lemma \ref{thm:marginal-delta} and $\delta_{\crv,t}=\dot{\crv}_{2}(t)\int_{0}^{\crv_{1}(r)}\delta_{(\xi,\crv_{2}(t))}d\xi$,
we see
\[
\lim_{j',j''}\left\langle \bS_{j''}\cE_{\crv}\bOne_{[s,r]},\bS_{j'}\delta_{\crv,r}\right\rangle =\frac{1}{2}\dot{\crv}_{2}(r)\crv_{1}(r).
\]
Hence by Lemma \ref{thm:|<S_jcEchi_[s,r1],f_{r2}>|<C}, and the dominated
convergence, 
\[
\lim_{j',j''}\int_{s}^{t}\left\langle \bS_{j''}\cE_{\crv}\bOne_{[s,r]},\bS_{j'}\delta_{\crv,r}\right\rangle dr=\frac{1}{2}\int_{s}^{t}\dot{\crv}_{2}(r)\crv_{1}(r)dr\ 
\]
and hence (\ref{eq:lim_{j,j'}I_1=00003D0}) holds.\end{proof}
\begin{lem}
\label{thm:lim_{j,j'}I_2=00003D0}\thlab{thm:lim\_\{j,j'\}I\_2=0}Define
$I_{2}$ by (\ref{eq:def:I_2}). Then for any $\crv\in\Crv$, $\lim_{j,j'\to\infty}I_{2}=0$.\end{lem}
\begin{proof}
Suppose $j<j'$. Let

\[
\symb{R_{j,j'}(x)}:=\left\langle \chchi_{j,j'},\chchi_{j,j'}\left(\cdot-x\right)\right\rangle ,\quad x\in\R^{2}.
\]
Then we have

\begin{align}
 & \left\langle \bS_{j,j'}\delta_{\crv,r},\bS_{j,j'}\delta_{\crv,r'}\right\rangle \nonumber \\
 & =\left\langle \bS_{j,j'}\dot{\crv}_{2}(r)\int_{0}^{\crv_{1}(r)}\delta_{(x_{1},\crv_{2}(r)}dx_{1},\ \bS_{j,j'}\dot{\crv}_{2}(r')\int_{0}^{\crv_{1}(r')}\delta_{(x_{1}',\crv_{2}(r')}dx_{1}'\right\rangle \nonumber \\
 & =\dot{\crv}_{2}(r)\int_{0}^{\crv_{1}(r)}\dot{\crv}_{2}(r')\int_{0}^{\crv_{1}(r')}\left\langle \bS_{j,j'}\delta_{(x_{1},\crv_{2}(r))},\ \bS_{j,j'}\delta_{(x_{1}',\crv_{2}(r'))}\right\rangle dx_{1}'dx_{1}\ \nonumber \\
 & =\dot{\crv}_{2}(r)\int_{0}^{\crv_{1}(r)}\dot{\crv}_{2}(r')\int_{0}^{\crv_{1}(r')}\left\langle \bS_{j,j'}\delta,\ \bS_{j,j'}\delta_{(x_{1}',\crv_{2}(r'))-(x_{1},\crv_{2}(r))}\right\rangle dx_{1}'dx_{1}\nonumber \\
 & =\dot{\crv}_{2}(r)\int_{0}^{\crv_{1}(r)}dx_{1}\dot{\crv}_{2}(r')\int_{0}^{\crv_{1}(r')}dx_{1}'\left\langle \chchi_{j,j'},\chchi_{j,j'}\left(\cdot-(x_{1}',\crv_{2}(r'))+(x_{1},\crv_{2}(r)\right)\right\rangle \nonumber \\
 & =\dot{\crv}_{2}(r)\int_{0}^{\crv_{1}(r)}\dot{\crv}_{2}(r')\int_{0}^{\crv_{1}(r')}R_{j,j'}\left(-(x_{1}',\crv_{2}(r'))+(x_{1},\crv_{2}(r))\right)dx_{1}'dx_{1}\ \label{eq:=00003DcintcintR}
\end{align}
Let
\begin{align*}
 & \symb{F_{j,j',r'}(r)}:=\left\langle \bS_{j'}\cE_{\crv}\bOne_{[s,r]},\bS_{j'}\cE_{\crv}\bOne_{[s,r']}\right\rangle ,\quad\\
 & \symb{R_{j,j'}'(x)}:=R_{j,j'}\left(-(x_{1}',\crv_{2}(r'))+(x_{1},x_{2})\right).
\end{align*}
 Then from (\ref{eq:=00003DcintcintR}) we have
\begin{align*}
I_{2} & =\int_{s}^{t}\int_{s}^{t}F_{j,j',r'}(r)\dot{\crv}_{2}(r)\int_{0}^{\crv_{1}(r)}\dot{\crv}_{2}(r')\int_{0}^{\crv_{1}(r')}R_{j,j'}'(x_{1},\crv_{2}(r))dx_{1}'\, dx_{1}dr'\, dr\ \\
 & =\int_{s}^{t}\dot{\crv}_{2}(r')\int_{0}^{\crv_{1}(r')}\int_{s}^{t}\dot{\crv}_{2}(r)\int_{0}^{\crv_{1}(r)}F_{j,j',r'}(r)R_{j,j'}'(x_{1},\crv_{2}(r))dx_{1}dr\, dx_{1}'dr'\\
 & =\int_{s}^{t}\dot{\crv}_{2}(r')\int_{0}^{\crv_{1}(r')}\hat{\cE}_{\crv}(F_{j,j',r'},R_{j,j'}')dx_{1}'dr'\\
 & =\int_{s}^{t}\dot{\crv}_{2}(r')\int_{0}^{\crv_{1}(r')}\left\langle R_{j,j'}',\ \cE_{\crv}F_{j,j',r'}\right\rangle dx_{1}'dr'\\
 & =\left\langle R_{j,j'},\ \int_{s}^{t}\dot{\crv}_{2}(r')\int_{0}^{\crv_{1}(r')}\left(\tau_{(x_{1}',\crv_{2}(r'))}\cE_{\crv}F_{j,j',r'}\right)dx_{1}'dr'\right\rangle \ 
\end{align*}
where
\[
\symb{(\tau_{x}f)(y)}:=f(y+x).
\]
Notice the fact that for any function $G\in C(\R^{2})$ with compact
support, $\lim_{j,j'\to\infty}\left\langle R_{j,j'},G\right\rangle =0$
holds. We see that the function
\[
\R^{2}\ni x\mapsto\int_{s}^{t}\dot{\crv}_{2}(r')\int_{0}^{\crv_{1}(r')}\left(\tau_{(x_{1}',\crv_{2}(r'))}\cE_{\crv}F_{j,j',r'}\right)(x)dx_{1}'dr'
\]
is continuous, and its support is compact. Thus we have
\[
\lim_{j,j'\to\infty}I_{2}=0.
\]
\end{proof}
\begin{lem}
\label{thm:lim_{j,j'}I_3=00003D0}\thlab{thm:lim\_\{j,j'\}I\_3=0}Define
$I_{3}$ by (\ref{eq:def:I_3}). Then for any $\crv\in\Crv_{\Rot}$,
$\lim_{j,j'\to\infty}I_{3}=0$.\end{lem}
\begin{proof}
By Lemma \ref{thm:|<S_jcEchi_[s,r1],f_{r2}>|<C} with $f_{\crv,r}^{j}=\bS_{j}\delta_{\crv,r}$,
there exists $C=C(\crv)>0$ such that for all $j,j'\ge-1$ and $r,r'\in[s,t]$,
\[
\left|\left\langle \bS_{j'}\cE_{\crv}\bOne_{[s,r]},\bS_{j,j'}\delta_{\crv,r'}\right\rangle \right|\le\left|\left\langle \bS_{j'}\cE_{\crv}\bOne_{[s,r]},\ \bS_{j}\delta_{\crv,r'}\right\rangle \right|+\left|\left\langle \bS_{j'}\cE_{\crv}\bOne_{[s,r]},\ \bS_{j'}\delta_{\crv,r'}\right\rangle \right|<C.
\]
By Lemmas \ref{thm:cE1[s,t]-finiteCombi} and \ref{thm:marginal-delta},
we find that for almost all $r,r'\in[s,t]$ and $x_{1},x_{1}'\in\R$,
\begin{align*}
\lim_{j',j\to\infty}\left\langle \bS_{j}\cE_{\crv}\bOne_{[s,r]},\ \bS_{j'}\delta_{(x_{1},\crv_{2}(r'))}\right\rangle  & =\lim_{j'}\left\langle \bS_{j'}\cE_{\crv}\bOne_{[s,r]},\ \bS_{j'}\delta_{(x_{1},\crv_{2}(r'))}\right\rangle \ \\
 & =\left(\cE_{\crv}\bOne_{[s,r]}\right)(x_{1},\crv_{2}(r')),
\end{align*}
and hence
\[
\lim_{j',j\to\infty}\left\langle \bS_{j,j'}\cE_{\crv}\bOne_{[s,r]},\ \bS_{j'}\delta_{(x_{1},\crv_{2}(r'))}\right\rangle =0.
\]
Thus, by $\delta_{\crv,t}=\dot{\crv}_{2}(t)\int_{0}^{\crv_{1}(r)}\delta_{(\xi,\crv_{2}(t))}d\xi$
and the dominated convergence, we have
\begin{align*}
 & \lim_{j,j'\to\infty}I_{3}\\
 & =\lim_{j,j'\to\infty}\delta_{kl}\int_{s}^{t}\int_{s}^{t}\dot{\crv}_{2}(r')\int_{0}^{\crv_{1}(r')}\dot{\crv}_{2}(r)\int_{0}^{\crv_{1}(r)}\left\langle \bS_{j'}\cE_{\crv}\bOne_{[s,r]},\ \bS_{j,j'}\delta_{(x_{1},\crv_{2}(r'))}\right\rangle \ \\
 & \quad\times\left\langle \bS_{j,j'}\delta_{(x_{1},\crv_{2}(r))},\ \bS_{j'}\cE_{\crv}\bOne_{[s,r']}\right\rangle dx_{1}dx_{1}'dr'dr\\
 & =0.
\end{align*}
\end{proof}
\begin{lem}
\label{thm:lim_{j,j'}N_1=00003D0}Define $N_{1}$ by (\ref{eq:def:N_1}).
Then for any $\crv\in\Crv_{\Rot}$, $\lim_{j,j'\to\infty}N_{1}=0$.\end{lem}
\begin{proof}
Follows from Lemmas \ref{thm:N_1^2=00003DI_1^2+I_2+I_3-jj'}, \ref{thm:lim_{j,j'}I_1=00003D0},
\ref{thm:lim_{j,j'}I_2=00003D0}, and \ref{thm:lim_{j,j'}I_3=00003D0}.\end{proof}
\begin{lem}
\label{thm:N_2^2=00003DJ_1^2+J_2+J_3}For any $\crv\in\Crv$,
\[
N_{2}^{2}=J_{1}^{2}+J_{2}+J_{3}
\]
where
\begin{align*}
 & \symb{J_{1}}:=\delta_{kl}\int_{s}^{t}\left\langle \bS_{j}\delta_{\crv,r},\ \bS_{j,j'}\cE_{\crv}\bOne_{[s,r]}\right\rangle dr,\\
 & \symb{J_{2}}:=\int_{s}^{t}\int_{s}^{t}\left\langle \bS_{j}\delta_{\crv,r},\ \bS_{j}\delta_{\crv,r'}\right\rangle \left\langle \bS_{j,j'}\cE_{\crv}\bOne_{[s,r]},\ \bS_{j,j'}\cE_{\crv}\bOne_{[s,r']}\right\rangle dr'dr\\
 & \symb{J_{3}}:=\delta_{kl}\int_{s}^{t}\int_{s}^{t}\left\langle \bS_{j}\delta_{\crv,r},\ \bS_{j,j'}\cE_{\crv}\bOne_{[s,r']}\right\rangle \left\langle \bS_{j,j'}\cE_{\crv}\bOne_{[s,r]},\ \bS_{j}\delta_{\crv,r'}\right\rangle dr'dr
\end{align*}
\end{lem}
\begin{proof}
By a straightforward calculation.%
\end{proof}
\begin{lem}
\label{thm:|int_s^tdr|S_jdelta|<|}For any $\crv\in\Crv_{\infty}$,
there exists $C=C(\crv)$ such that for all $j$ and $0\le s<t\le1$,

\[
\left\Vert \int_{s}^{t}\left|\bS_{j}\delta_{\crv,r}\right|dr\right\Vert _{L^{2}(\R^{2})}\le C.
\]
\end{lem}
\begin{proof}
Let
\begin{align*}
 & \symb{H_{j,y}(x)}:=\left|\left(\bS_{j}\delta_{x}\right)(y)\right|=\left|\left(\bS_{j}\delta\right)(y-x)\right|,\quad x,y\in\R^{2}.\\
 & \symb{s_{[s,t]}(r)}:=\sgn(\dot{\crv}_{2}(r))\bOne_{[s,t]}(r)
\end{align*}
Then we have
\begin{align*}
\int_{s}^{t}\left|\left(\bS_{j}\delta_{\crv,r}\right)(y)\right|dr & \le\int_{s}^{t}\left|\dot{\crv}_{2}(r)\right|\int_{0}^{\crv_{1}(r)}\left|\left(\bS_{j}\delta_{(x_{1},\crv_{2}(r))}\right)(y)\right|dx_{1}dr\ \\
 & =\int_{0}^{1}\dot{\crv}_{2}(r)\int_{0}^{\crv_{1}(r)}\left|\left(\bS_{j}\delta_{(x_{1},\crv_{2}(r))}\right)(y)\right|\sgn(\dot{\crv}_{2}(r))\bOne_{[s,t]}(r)dx_{1}dr\\
 & =\hat{\cE}_{\crv}(H_{j,y},\ s_{[s,t]})=\left\langle H_{j,y},\ \cE_{\crv}s_{[s,t]}\right\rangle =\left(H_{j,0}*\cE_{\crv}s_{[s,t]}\right)(y)
\end{align*}
On the other hand we find
\begin{align*}
\left\Vert \cE_{\crv}s_{[s,t]}\right\Vert _{L^{2}} & \le\Leb\left(\supp\cE_{\crv}s_{[s,t]}\right)^{1/2}\left\Vert \cE_{\crv}s_{[s,t]}\right\Vert _{L^{\infty}}\\
 & \le\Leb\left(\supp\cE_{\crv}s_{[s,t]}\right)^{1/2}\left\Vert \cE_{\crv}\right\Vert _{\infty\infty}\le C_{1}(\crv).
\end{align*}
Thus
\begin{align*}
\left\Vert \int_{s}^{t}dr\left|\bS_{j}\delta_{\crv,r}\right|\right\Vert _{L^{2}} & \le\left\Vert H_{j,0}*\cE_{\crv}s_{[s,t]}\right\Vert _{L^{2}}\\
 & =\left\Vert \left(\bS_{j}\delta\right)*\cE_{\crv}s_{[s,t]}\right\Vert _{L^{2}}\\
 & \le\left\Vert \bS_{j}\delta\right\Vert _{L^{1}}\left\Vert \cE_{\crv}s_{[s,t]}\right\Vert _{L^{2}}\\
 & =\left\Vert \bS_{0}\delta\right\Vert _{L^{1}}\left\Vert \cE_{\crv}s_{[s,t]}\right\Vert _{L^{2}}\ \\
 & \le\left\Vert \bS_{0}\delta\right\Vert _{L^{1}}C_{1}(\crv)\le C_{2}(\crv).
\end{align*}
\end{proof}
\begin{lem}
\label{thm:lim_{jj'}N_2=00003D0}Define $N_{2}$ by (\ref{eq:def:N_2}).
Then for any $\crv\in\Crv_{\infty}$, $\lim_{j,j'\to\infty}N_{2}=0$.\end{lem}
\begin{proof}
By Lemma (\ref{thm:N_2^2=00003DJ_1^2+J_2+J_3}), it suffices to show
that

\[
\lim_{j,j'}J_{i}=0,\quad i=1,2,3,
\]
The proof of $\lim_{j,j'}J_{1}=0$ is similar to that of $\lim_{j,j'}I_{1}=0$.
The proof of $\lim_{j,j'}J_{3}=0$ is similar to $\lim_{j,j'}I_{3}=0$.
We will show $\lim_{j,j'\to\infty}J_{2}=0$. By Lemmas \ref{thm:B2infty^{1/2}(R^2)cond-leng}
and \ref{thm:|cE.chi[s,t]|_{B_{2,infty}}<}, for any $\besovorder\in(0,1/2]$
we have 

\[
\symb{\cN_{\crv,\besovorder}}:=\sup_{r\in[s,t]}\left\Vert \cE_{\crv}\bOne_{[s,r]}\right\Vert _{B_{2,\infty}^{\besovorder}(\R^{2})}<\infty
\]
Thus we have

\[
\left\Vert \bDelta_{j}\cE_{\crv}\bOne_{[s,r]}\right\Vert _{L^{2}(\R^{2})}\le\cN_{\crv,\besovorder}2^{-\besovorder j}
\]
and hence we find that if $j\le j'$,
\begin{align*}
 & \left\Vert \bS_{j,j'}\cE_{\crv}\bOne_{[s,r]}\right\Vert _{L^{2}(\R^{2})}=\Bigl\Vert\sum_{i=j}^{j'-1}\bDelta_{i}\cE_{\crv}\bOne_{[s,r]}\Bigr\Vert_{L^{2}(\R^{2})}\ \\
 & \le\sum_{i=j}^{j'-1}\left\Vert \bDelta_{i}\cE_{\crv}\bOne_{[s,r]}\right\Vert _{L^{2}(\R^{2})}\le\frac{\cN_{\crv,\besovorder}}{1-2^{-\besovorder}}2^{-\besovorder j}
\end{align*}
and so
\begin{align*}
\left|\left\langle \bS_{j,j'}\cE_{\crv}\bOne_{[s,r]},\bS_{j,j'}\cE_{\crv}\bOne_{[s,r']}\right\rangle \right| & \le\left\Vert \bS_{j,j'}\cE_{\crv}\bOne_{[s,r]}\right\Vert _{L^{2}(\R^{2})}\left\Vert \bS_{j,j'}\cE_{\crv}\bOne_{[s,r']}\right\Vert _{L^{2}(\R^{2})}\\
 & \le C2^{-2\besovorder j}.\ \ 
\end{align*}
Thus we have
\begin{align*}
\left|J_{2}\right| & =\left|\int_{s}^{t}\int_{s}^{t}\left\langle \bS_{j}\delta_{\crv,r},\bS_{j}\delta_{\crv,r'}\right\rangle \left\langle \bS_{j,j'}\cE_{\crv}\bOne_{[s,r]},\bS_{j,j'}\cE_{\crv}\bOne_{[s,r']}\right\rangle dr'dr\right|\\
 & \le\int_{s}^{t}\int_{s}^{t}\left|\left\langle \bS_{j}\delta_{\crv,r},\bS_{j}\delta_{\crv,r'}\right\rangle \left\langle \bS_{j,j'}\cE_{\crv}\bOne_{[s,r]},\bS_{j,j'}\cE_{\crv}\bOne_{[s,r']}\right\rangle \right|dr'dr\ \\
 & \le C2^{-2\besovorder j}\int_{s}^{t}\int_{s}^{t}\left|\left\langle \bS_{j}\delta_{\crv,r},\bS_{j}\delta_{\crv,r'}\right\rangle \right|dr'dr\\
 & =C2^{-2\besovorder j}\int_{s}^{t}\int_{s}^{t}\left\langle \left|\bS_{j}\delta_{\crv,r}\right|,\left|\bS_{j}\delta_{\crv,r'}\right|\right\rangle dr'dr\\
 & =C2^{-2\besovorder j}\left\Vert \int_{s}^{t}\left|\bS_{j}\delta_{\crv,r}\right|dr\right\Vert _{L^{2}(\R^{2})}^{2}\\
 & \le C_{2}2^{-2\besovorder j}\ 
\end{align*}
where the last inequality is by Lemma \ref{thm:|int_s^tdr|S_jdelta|<|}.
Thus we have shown $\lim_{j,j'\to\infty}J_{2}=0$. This completes
the proof.
\end{proof}

\section{Rough path convergence}
\begin{lem}
[Uniform rough path bounds in $L^p$] \label{thm:roughPathBoundsInLp}Let
$\crv\in\Crv_{\infty}$, $q\in[1,\infty)$ and $\al\in(1/3,1/2)$.
Then

\[
\sup_{j}\left\Vert d_{{\rm CC};\al\hHol:[0,1]}\bigl(\bX_{\crv}^{(j)},o\bigr)\right\Vert _{L^{q}(\Prob)}<\infty.
\]
\end{lem}
\begin{proof}
Notice that $d_{{\rm CC}}(\bX_{s}^{(j)},\bX_{t}^{(j)})\simeq\bigl|X_{t}^{(j)}-X_{s}^{(j)}\bigr|+\bigl|\bbX_{t}^{(j)}-\bbX_{s}^{(j)}-X_{s}^{(j)}\otimes(X_{t}^{(j)}-X_{s}^{(j)})\bigr|^{1/2}$.
Because $\left(1,X^{(j)},\bbX^{(j)}\right)=\Sn(X^{(j)})$ and $X_{0}^{(j)}=0$,
it follows from Chen's relation (Theorem \ref{thm:Chen's-relation})
that $\bbX_{s,t}^{(j)}=\bbX_{t}^{(j)}-\bbX_{s}^{(j)}-X_{s}^{(j)}\otimes\left(X_{t}^{(j)}-X_{s}^{(j)}\right)$.
Thus we see $d_{{\rm CC}}(\bX_{s}^{(j)},\bX_{t}^{(j)})\simeq\left|X_{s,t}^{(j)}\right|+\left|\bbX_{s,t}^{(j)}\right|^{1/2}$,
and hence

\begin{align*}
\left\Vert d_{{\rm CC}}(\bX_{s}^{(j)},\bX_{t}^{(j)})\right\Vert _{L^{q}(\Prob)} & \lesssim\left\Vert \bigl|X_{s,t}^{(j)}\bigr|\right\Vert _{L^{q}(\Prob)}+\left\Vert \bigl|\bbX_{s,t}^{(j)}\bigr|^{1/2}\right\Vert _{L^{q}(\Prob)}\ \\
 & =\left\Vert X_{s,t}^{(j)}\right\Vert _{L^{q}(\Prob,\fg)}+\left\Vert \bbX_{s,t}^{(j)}\right\Vert _{L^{q/2}(\Prob,\fg)}^{1/2}
\end{align*}
By Prop. \ref{thm:X^(j)-L^p-bdd} and Prop. \ref{thm:Ex[|bbX|^2]<},
we have for all $j\ge-1,\ \beta\in(0,1/2)$ and $\ q\in[1,\infty)$,
\[
\bigl\Vert X_{s,t}^{(j)}\bigr\Vert_{L^{q}(\Prob,\fg)}\le C_{1}\left|t-s\right|^{\beta},\quad\bigl\Vert\bbX_{s,t}^{(j)}\bigr\Vert_{L^{q}(\Prob,\fg)}\le C_{2}\left|t-s\right|^{2\beta},\quad\ 
\]
Hence there exists $C_{3}$ such that 
\[
\left\Vert d_{{\rm CC}}(\bX_{s}^{(j)},\bX_{t}^{(j)})\right\Vert _{L^{q}(\Prob)}\le C_{3}\left|t-s\right|^{\beta}\quad\forall j\ge-1,\ \beta\in(0,1/2),\ q\in[1,\infty)\ 
\]
For $0\le\orderb<\ordera$, let $C(\ordera,\orderb,T)$ be of Theorem
\ref{thm:G2process-Lq-bound} with $M=C_{3}$. Then we see 

\[
\left\Vert d_{{\rm CC};\al\hHol:[0,T]}\left(\bX^{(j)},o\right)\right\Vert _{L^{q}(\Prob)}\le C(\beta,\al,1)C_{3},\quad\forall j\ge-1,\ \al\in(0,\beta).
\]
This completes the proof.
\end{proof}
{}
\begin{lem}
[pointwise $L^{p}$ convergence] \label{thm:pointwise-Lp-convergence}For
each $p\in[1,\infty)$ and $0\le s<t\le1$, $\bX_{s,t}^{(j)}=\bigl(1,X_{s,t}^{(j)},\bbX_{s,t}^{(j)}\bigr)$
converges to an element $\symb{\bX_{s,t}}=\bigl(1,X_{s,t},\bbX_{s,t}\bigr)$
in $L^{p}$, that is, 

\[
\lim_{j}\bigl\Vert X_{s,t}-X_{s,t}^{(j)}\bigr\Vert_{L^{p}(\Prob,\fg)}=\lim_{j}\bigl\Vert\bbX_{s,t}-\bbX_{s,t}^{(j)}\bigr\Vert_{L^{p}(\Prob,\fg\otimes\fg)}=0
\]
hold. Equivalently, 

\[
\lim_{j\to\infty}\bigl\Vert d_{{\rm CC}}(\bX_{s,t}^{(j)},\bX_{s,t})\bigr\Vert_{L^{q}(\Prob)}=0.
\]
\end{lem}
\begin{proof}
The convergence of $\lim_{j}X_{s,t}^{(j)}$ in $L^{p}(\Prob,\fg)$
follows from Prop. \ref{thm:X^(j)-L^p-converge}. The convergence
of $\lim_{j}\bbX_{s,t}^{(j)}$ in $L^{p}(\Prob,\fg\otimes\fg)$ follows
from Prop. \ref{thm:lim_{jj'}|bbX-bbX|=00003D0}.\end{proof}
\begin{thm}
[rough path convergence in $L^p$] \label{thm:roughPathConvergeLp}\thlab{thm:roughPathConvergeLp}
Suppose $\crv\in\Crv_{\infty}$, $\rpHol\in(1/3,1/2)$, and $p\ge1$.
Let $\bX_{s,t}=\lim_{j}\bX_{s,t}^{(j)}$ be given by Lemma \ref{thm:pointwise-Lp-convergence},
and $\symb{\bX_{t}}:=\bX_{0,t}=\bigl(1,X_{t},\bbX_{t}\bigr)$. Then
$\bX$ is a weak geometric $\rpHol$-H\"older rough path, i.e. $\bX\in C^{\rpHol\hHol}\bigl([0,1],G^{(2)}(\fg)\bigr)$,
and $\bX^{(j)}\to\bX$ in $C^{\rpHol\hHol}\bigl([0,1],G^{(2)}(\fg)\bigr)$
and $L^{p}(\Prob)$, i.e.
\[
\lim_{j\to\infty}\left\Vert d_{{\rm CC},\rpHol\hHol;[0,1]}\bigl(\bX,\bX^{(j)}\bigr)\right\Vert _{L^{p}(\Prob)}=0.
\]
\end{thm}
\begin{proof}
This immediately follows from Prop. \ref{thm:roughPathBoundsInLp},
Prop. \ref{thm:pointwise-Lp-convergence}, and Theorem \ref{thm:roughPath-Lq-convergence-G2}.\end{proof}
\begin{cor}
\label{thm:roughPathConverge-ae}\thlab{thm:roughPathConverge-ae}
Suppose $\crv\in\Crv_{\infty}$, $\rpHol\in(1/3,1/2)$. Then if $n:\N\to\N$
increases rapidly enough, 
\[
\Prob\left[\lim_{k\to\infty}d_{{\rm CC},\rpHol\hHol;[0,1]}\bigl(\bX,\bX^{(n(k))}\bigr)=0\right]=1.
\]

\end{cor}
{}

Now the ODE (\ref{eq:ptrans-normalform}) for the $j$th \term{approximate holonomy}
$\symb{\ptrans_{\crv,A}^{(j)}}$ associated with $W^{(j)}$ is written
as

\[
d\ptrans_{\crv,A}^{(j)}=\cV(\ptrans_{\crv,A}^{(j)})dX_{\crv}^{(j)},\quad\ptrans_{\crv,A}^{(j)}(0)=1_{G}\in G.
\]
Recall that $X_{\crv}^{(j)}$ is expressed by $W^{(j)}$ by (\ref{eq:X^(j)=00003DcE(W,1)}).

\begin{thm}
\label{thm:aeRoughPathConverge} For any countable subset $\Gamma\subset\Crv_{\infty}$,
and $n:\N\to\N$ increasing rapidly enough,
\[
\Prob\Bigl[\forall\crv\in\Gamma,\ \symb{\ptrans_{\crv}^{(\infty)}}:=\lim_{k\to\infty}\ptrans_{\crv}^{(n(k))}\,(\text{uniform})\in C([0,1],G)\Bigr]=1.
\]
Moreover, for, $\rpHol\in(1/3,1/2)$, $\lift(\ptrans_{\crv}^{(n(k))})$
converges to $\symb{\hat{\ptrans}_{\crv}^{(\infty)}}=(1,\hat{\ptrans}_{\crv}^{(\infty)[1]},\hat{\ptrans}_{\crv}^{(\infty)[2]})\in C^{\rpHol\hHol}([0,1],G^{(2)}(\Mat))$
a.s., where $\hat{\ptrans}_{\crv}^{(\infty)[1]}=\ptrans_{\crv}^{(\infty)}$.
That is,
\[
\Prob\Bigl[\forall\crv\in\Gamma,\ \lim_{k\to\infty}d_{{\rm CC},\rpHol\hHol;[0,1]}\left(\hat{\ptrans}_{\crv}^{(\infty)},\lift(\ptrans_{\crv}^{(n(k))})\right)=0\Bigr]=1.
\]
\end{thm}
\begin{proof}
Note that if we let $n_{i}:\N\to\N$ be increasing for each $i\in\N$,
then $n(k):=\max_{1\le i\le n}n_{i}(k)\ (k\in\N)$ increases more
rapidly than each $n_{i}$. Thus the theorem follows from Theorems
\ref{thm:Existence-of-RDE-sol}, \ref{thm:Existence-of-RDE-sol-full},
\ref{thm:fullRDE-conti} and Corollary \ref{thm:roughPathConverge-ae}.
\end{proof}
We call $\ptrans_{\crv}^{(\infty)}(1)$ the \term{holonomy-valued random variable}
(or simply the \term{holonomy variable}) along $\crv\in\Lasso\cap\Crv_{\infty}$.

\section{Wilson loop}

\label{sec:wilson-loop}

The \term{law of Wilson loops in  the YM theory} on $\R^{2}$ (with
the usual Euclidean metric) is described as follows (e.g. \cite{Lev03}):
Let $\cL$ be a set of lassos with some regularity condition. Then 

(i) The Wilson loop $\ptrans_{\crv}(1)$ is independent of $\ptrans_{\crv'}(1)$
if $\crv,\crv'\in\cL$ and $D(\crv)^{\circ}\cap D(\crv')^{\circ}=\emptyset$ 

(ii) The density $\rho$ of the Wilson loop $\ptrans_{\crv^{k}}(1)$
on $G$ with respect to Haar measure $dg$ is given by $\rho(g)=Q_{\Leb(D(\crv))}(g)$,
where $\symb{Q_{t}(x)}$ ($t\ge0$) denotes the fundamental solution
to the heat equation on the group $G$. 

In this section we show that holonomy variables $\ptrans_{\crv^{k}}^{(\infty)}$
given by Theorem \ref{thm:aeRoughPathConverge} obey the law the Wilson
loops in  the YM theory on $\R^{2}$.

Recall that $\symb{\fD}$ is the set of subsets $D\subset\R^{2}$
such that there exists a simple loop $\crv\in\Crv$ enclosing $D$,
and that $\rectlikedomain_{1}$ is the set of $E\in\fD$ such that
$E$ is convex w.r.t. $x_{1}$ (see (\ref{eq:def:rectlike})).

We use the following lemma in the proof of Theorem \ref{thm:wilson-indep-gauss-rect}.
\begin{lem}
{\rm{}\cite[Lemma 3.2.3]{Sen92}} \label{thm:gMg-indep}Let $M:\samples\to\fg$
be a random variable, $\Sigma$ a $\sigma$-algebra of measurable
subsets of $\samples$, and $g:\samples\to G$ a random variable which
is measurable with respect to $\Sigma$. Assume that $M$ is independent
of $\Sigma$ and that the distribution of $M$ is the same as that
of $xMx^{-1}$ for every $x\in G$. Then the $\fg$-valued random
variable $gMg^{-1}$ is independent of $\Sigma$ and has the same
distribution as $M$.
\end{lem}
{}

If $E$ is a measurable subset of $\R^{2}$ then $\symb{\tau(E)}$
will denote the $\sigma$-algebra generated by all the random variables
$W(E')$ as $E'$ runs over the measurable subsets of $E$.
\begin{thm}
\label{thm:wilson-indep-gauss-rect}Let $\crv\in\Crv_{\infty}\cap\Lasso(x)$
satisfy $D(\crv)\in\rectlikedomain_{1}$. Then 

(i) The $G$-valued random variable $\ptrans_{\crv}^{(\infty)}(1)$
is independent of the $\sigma$-algebra $\tau(\R^{2}\setminus D(\crv))$.

(ii) The density $\rho$ of the Wilson loop $\ptrans_{\crv}^{(\infty)}(1)$
on $G$ with respect to Haar measure $dg$ is given by $\rho(g)=Q_{\Leb(D(\crv))}(g).$
In other words,
\[
\Ex\left[f(\ptrans_{\crv}^{(\infty)}(1))\right]=\int_{G}f(g)Q_{\Leb(D(\crv))}(g)dg.
\]
for every bounded Borel function $f$ on $G$. \end{thm}
\begin{proof}
The proof of (i) is similar to that of \cite[Lemma 3.2.6]{Sen92},
and the proof of (ii) is to that of \cite[Theorem 3.2.10]{Sen92}
(see also \cite{Dri89}), and so we will give only a sketch. 

(i) In the settings of Sec. \ref{sec:axial-gauge}, let $F_{12}=W^{(j)}$,
and denote the corresponding $F_{t}^{D}$, $B_{t}^{D}$ and $U$ by
$\symb{F_{t}^{D,(j)}}$, $\symb{B_{t}^{D,(j)}}$ and $\symb{U^{(j)}}$,
respectively. Let
\begin{equation}
\symb{F_{t}^{D,(\infty)}}:=\lim_{j\to\infty}F_{t}^{D,(j)}=W(D_{t}),\quad\symb{B_{t}^{D,(\infty)}}:=\lim_{j\to\infty}B_{t}^{D,(j)}\label{eq:def:B^(D,infty)}
\end{equation}
Let us write $B_{t}^{D,(\infty)}$ as a formal integral
\begin{equation}
B_{t}^{D,(\infty)}=\int_{a}^{t}\ptrans_{\crv^{1}}^{(\infty)}(s)^{-1}dF_{s}^{D,(\infty)}\ptrans_{\crv^{1}}^{(\infty)}(s).\label{eq:B^(D,infty)-int}
\end{equation}
We see that $F_{t}^{D,(\infty)}$is a $t$-reparametrization of a
standard $\fg$-valued Brownian motion such that
\[
\Ex\Bigl[\bigl\Vert F_{t}^{D,(\infty)}\bigr\Vert_{\HS}^{2}\Bigr]=\Leb(D_{t}).
\]
Hence the formal integral (\ref{eq:B^(D,infty)-int}) can be justified
as a rough integral for Brownian rough paths \cite{FH14}, and also
as a stochastic integral in the Stratonovich sense. Thus we see that
$B_{t}^{D,(\infty)}$ is also $t$-reparametrization of a standard
$\fg$-valued Brownian motion with $\Ex\Bigl[\bigl\Vert B_{t}^{D,(\infty)}\bigr\Vert_{\HS}^{2}\Bigr]=\Leb(D_{t})$.
By Theorem \ref{thm:aeRoughPathConverge}, we see $B_{t}^{D,(n(k))}\to B_{t}^{D,(\infty)}$
as $k\to\infty$ uniformly a.s., if $n:\N\to\N$ increases rapidly
enough; Moreover we find that $\lift(B_{t}^{D,(n(k))})$ converges
to $\symb{\bB^{D,(\infty)}}=(1,B^{D,(\infty)},\bbB^{D,(\infty)})$
in $C^{\rpHol\hHol}([0,1],G^{(2)}(\Mat))$. 

The ODE (\ref{eq:dU=00003D-UdB}) is now written as
\begin{align}
dU^{(j)}(t)= & -U^{(j)}(t)dB_{t}^{D,(j)}.\quad
\end{align}
By Theorem \ref{thm:aeRoughPathConverge} and \ref{thm:Existence-of-RDE-sol},
we find that $\symb{U^{(\infty)}}:=\pi\left(0,I;-\bB^{D,(\infty)}\right)$
is well-defined, that is, the solution of the RDE
\begin{align}
dU^{(\infty)}(t)=-U^{(\infty)}(t)d\bB_{t}^{D,(\infty)},\label{eq:dU^infty=00003D-U^inftydbB}
\end{align}
uniquely exists. Since $F_{t}^{D,(\infty)}$ is independent of $\tau(\R^{2}\setminus D(\crv))$,
we see from (\ref{eq:def:B^(D,infty)}) and Lemma \ref{thm:gMg-indep}
that $B_{t}^{D,(\infty)}$ is independent of $\tau(\R^{2}\setminus D(\crv))$,
and so is $\bB^{D,(\infty)}$. Hence $U^{(\infty)}(t)$, especially
$\ptrans_{\crv}(1)=U^{(\infty)}(1)$, is also independent of $\tau(\R^{2}\setminus D(\crv))$. 

(ii) Since $B_{t}^{D,(\infty)}$ is a reparametrization of a standard
$\fg$-valued Brownian motion with $\Ex\bigl[\bigl\Vert B_{t}^{D,(\infty)}\bigr\Vert_{\HS}^{2}\bigr]=\Leb(D_{t})$,
Eq. (\ref{eq:dU^infty=00003D-U^inftydbB}) leads to the Stratonovich
SDE
\[
dU^{(\infty)}(t)=-U^{(\infty)}(t)\circ dB_{t}^{D,(\infty)},
\]
which implies that $U^{(\infty)}(t)$ is a $t$-reparametrization
of a $G$-valued Brownian motion with density $Q_{\Leb(D_{t})}$.
Thus the Wilson loop $\ptrans_{\crv}^{(\infty)}(1)=U^{(\infty)}(1)$
has the density $Q_{\Leb(D_{1})}=Q_{\Leb(D(\crv))}$.
\end{proof}
Let $\symb{\rectlikedomain_{1,{\rm fin}}}$be the family of the finite
unions of sets in $\rectlikedomain_{1}$ which is : $\symb{\rectlikedomain_{1,{\rm fin}}}:=\{\bigcup_{k=1}^{n}D_{k};\ D_{k}\in\rectlikedomain_{1},1\le k\le n\in\N\}$.
\begin{cor}
Let $\crv\in\Crv_{\infty}\cap\Lasso(x)$ satisfy $D(\crv)\in\rectlikedomain_{1,{\rm fin}}$.
Then (i) and (ii) in Theorem \ref{thm:wilson-indep-gauss-rect} hold.\end{cor}
\begin{proof}
Follows from Lemma \ref{thm:ptrans-finite-lassos-prod}.\end{proof}
\begin{cor}
Let $\crv^{1},\crv^{2},...\in\Crv_{\infty}\cap\Lasso$, and suppose
that $D(\crv^{k})\in\rectlikedomain_{1,{\rm fin}}$ for all $k\in\N$,
and $D(\crv^{k})^{\circ}\cap D(\crv^{l})^{\circ}=\emptyset$ for $k\neq l$.
Then the Wilson loop $\ptrans_{\crv^{k}}^{(\infty)}(1)$ is independent
of $\ptrans_{\crv^{l}}^{(\infty)}(1)$ if $k\neq l$, and has the
density $Q_{\Leb(D(\crv^{k}))}$.
\end{cor}
{}

Our results are summarized as follows:
\begin{thm}
\label{thm:main-summ}Let $\crv^{1},\crv^{2},...\in\Crv_{\infty}\cap\Lasso$,
and suppose that $D(\crv^{k})\in\rectlikedomain_{1,{\rm fin}}$ for
all $k\in\N$. Then there exists a probability space $(\samples,\Prob)$
and a sequence of $\Omega^{1}(\R^{2},\fg)$-valued random variables
$A^{(n)}$ such that 
\[
\Prob\Bigl[\forall i\in\N,\ \ptrans_{\crv^{i}}:=\lim_{n\to\infty}\ptrans_{\crv^{i},A^{(n)}}\text{ (uniform) }\in C([0,1],G)\Bigr]=1,\ 
\]
and the set of the $G$-valued random variables $\{\ptrans_{\crv^{i}}\}_{i\in\N}$
obeys the law the Wilson loops in  the YM theory on $\R^{2}$.
\end{thm}

\section{Open problems}
\begin{conjecture}
\label{conj:main}Let $\Crv_{*}$ denote one of $\Crv_{\infty}$,
$\Crv_{\Rot}$, $\Crv$ and $C^{1\hvar}$ (continuous curves of bounded
variation). There exists a probability space $(\samples,\Prob)$ and
a sequence of $\Omega^{1}(\R^{2},\fg)$-valued random variables $A^{(n)}$
such that 
\[
\Prob\Bigl[\forall\crv\in\Crv_{*},\ \ptrans_{\crv}:=\lim_{n\to\infty}\ptrans_{\crv,A^{(n)}}\text{ (uniform) }\in C([0,1],G)\Bigr]=1,\ 
\]
and the set of the holonomy variables $\{\ptrans_{\crv}(1):\ \crv\in\Crv_{*}\cap\Lasso\}$
obeys the law the Wilson loops in the YM theory on $\R^{2}$.
\end{conjecture}
{}

This conjecture seems plausible for $\Crv_{*}=\Crv_{\infty},\Crv_{\Rot}$,
but the plausibility is obscurer for $\Crv_{*}=\Crv,\ C^{1\hvar}$.
If the conjecture is the case, the following question will arise:
\begin{problem}
Does the mapping $\Crv_{*}\ni\crv\mapsto\ptrans_{\crv}$ given in
Conj. \ref{conj:main} have any continuity in the sense of rough paths? 
\end{problem}
{}

This continuity is desirable to establish the notion of `rough gauge
fields.' However, thus far, we have no positive evidence of this continuity.

The method of \cite{Dri89,Sen92,Sen93,Sen97} strongly depend on special
gauge fixing (axial gauge in \cite{Dri89}, radial gauge in \cite{Sen92,Sen93,Sen97}),
and seem difficult to be generalized to other gauges; Generally, the
notions of gauge transformation and gauge symmetry are usually defined
on the \emph{classical} level (in terms of differential geometry),
and the rigorous treatment of those notions is more difficult in the
\emph{quantum} level. Although in this paper we confined ourselves
to the case of axial gauge, we conjecture that our method can be generalized
to other gauges, simply because a quantum gauge field can be approximated
by a classical (smooth) gauge fields in our method.

\section*{Acknowledgement}

The author thanks Professor Yuzuru Inahama of Kyushu University for
valuable advices. 

%\bibliographystyle{alpha}
%\bibliography{/home/leibniz/art/ybib}

\begin{thebibliography}{MWX16}

\bibitem[BCD11]{BCD11}
H.~Bahouri, J.-Y. Chemin, and R.~Danchin.
\newblock {\em {F}ourier analysis and nonlinear partial differential
  equations}.
\newblock Springer, Berlin, 2011.

\bibitem[CC13]{CC13}
R.~Catellier and K.~Chouk.
\newblock Paracontrolled distributions and the 3-dimensional stochastic
  quantization equation.
\newblock arXiv:1310.6869, 2013.

\bibitem[Dri89]{Dri89}
B.~K. Driver.
\newblock {YM}${}_2$: Continuum expectations, lattice convergence, and lassos.
\newblock {\em Commun. Math. Phys.}, 123:575--616, 1989.

\bibitem[FH14]{FH14}
P.~Friz and M.~Hairer.
\newblock {\em A Course on Rough Paths}.
\newblock Springer, Berlin, 2014.

\bibitem[FV10]{FV10b}
P.~Friz and N.~Victoir.
\newblock {\em Multidimensional Stochastic Processes as Rough Paths}.
\newblock Cambridge Studies in Advanced Mathematics. Cambridge University
  Press, Cambridge, 2010.

\bibitem[GIP15]{GIP15}
M.~Gubinelli, P.~Imkeller, and N.~Perkowski.
\newblock Paracontrolled distributions and singular {PDE}s.
\newblock {\em Forum of Mathematics, Pi}, 3(6):1--75, 2015.
\newblock arXiv:1210.2684.

\bibitem[Gra09]{Gra09}
L.~Grafakos.
\newblock {\em Modern {F}ourier Analysis}.
\newblock Springer, Berlin, second edition, 2009.

\bibitem[Hai14]{Hai14}
M.~Hairer.
\newblock A theory of regularity structures.
\newblock {\em Invent. Math.}, 198(2):269--504, 2014.
\newblock arXiv:1303.5113.

\bibitem[Hai15]{Hai15}
M.~Hairer.
\newblock Regularity structures and the dynamical $\phi_3^4$ model.
\newblock arXiv:1508.05261, 2015.

\bibitem[Jan97]{Jan97}
S.~Janson.
\newblock {\em {G}aussian {H}ilbert spaces}.
\newblock Cambridge University Press, Cambridge, 1997.

\bibitem[L{\'e}v03]{Lev03}
T.~L{\'e}vy.
\newblock {\em The {Y}ang-{M}ills measure for compact surfaces}, volume 166 of
  {\em Memoirs Amer. Math. Soc.}
\newblock American Mathematical Society, Providence, 2003.

\bibitem[MW16]{MW16}
J.-C. Mourrat and H.~Weber.
\newblock Global well-posedness of the dynamic {$\Phi_3^4$} model on the torus.
\newblock arXiv:1601.01234, 2016.

\bibitem[MWX16]{MWX16}
J.-C. Mourrat, H.~Weber, and W.~Xu.
\newblock Construction of {$\Phi_3^4$} diagrams for pedestrians.
\newblock arXiv:1610.08897, 2016.

\bibitem[Sen92]{Sen92}
A.~Sengupta.
\newblock The {Y}ang--{M}ills measure for {$S^2$}.
\newblock {\em J. Funct. Anal.}, 108:231--273, 1992.

\bibitem[Sen93]{Sen93}
A.~Sengupta.
\newblock Quantum gauge theory on compact surfaces.
\newblock {\em Ann. Phys. (NY)}, 221:17--52, 1993.

\bibitem[Sen97]{Sen97}
A.~Sengupta.
\newblock {\em Gauge Theory on Compact Surfaces}, volume 126 of {\em Memoirs of
  the Amer. Math. Soc.}
\newblock American Mathematical Society, Providence, 1997.

\bibitem[Tar07]{Tar07}
L.~Tartar.
\newblock {\em An Introduction to {S}obolev Spaces and Interpolation Spaces}.
\newblock Springer, Berlin, 2007.

\end{thebibliography}
\providecommand{\noopsort}[1]{}\providecommand{\singleletter}[1]{#1}%

%<index label>

%<index label dense>
\end{document}